\documentclass[12pt]{article}
\usepackage{latexsym}
\usepackage{amssymb}
\usepackage{amsmath}
\usepackage{mathtools}
\usepackage{amsthm}
\usepackage{amssymb}
\usepackage{graphicx}
\usepackage{caption}
\usepackage{epsfig}
\usepackage{float}
\usepackage{enumitem}

\usepackage{xcolor}

\newcommand{\ee}{{\bf e}}
\newcommand{\ha}{\textstyle{\frac{1}{2}}\displaystyle}
\newcommand{\qu}{\textstyle{\frac{1}{4}}\displaystyle}

\newcommand{\RR}{\mathbb R}
\newcommand{\ii}{{\rm i}}
\newcommand{\sF}{{\scriptscriptstyle F}} 
\newcommand{\sP}{{\scriptscriptstyle P}}
\newcommand{\sV}{{\scriptscriptstyle V}} 
\newcommand{\sT}{{\scriptscriptstyle T}}
\newcommand{\sC}{{\scriptscriptstyle C}}

\newtheorem{theorem}{Theorem}[section]
\newtheorem*{prop*}{Proposition}
\newtheorem{prop}[theorem]{Proposition}

\newtheorem{lemma}[theorem]{Lemma}
\newtheorem*{slemma}{\bf Separating $2$-jet Lemma}

\newtheorem{coGspecial}[theorem]{}

\title{Contact with circles and Euclidean invariants of smooth surfaces in $\RR^3$}
\author{Peter Giblin \and Graham Reeve \and Ricardo Uribe-Vargas}
\date{}
\begin{document}

\theoremstyle{remark}
\newtheorem{rem}[theorem]{\bf Remark}
\newtheorem*{remark*}{\bf Remark}
\newtheorem{rems}[theorem]{\bf Remarks}

\theoremstyle{definition}
\newtheorem{defs}[theorem]{\bf Definition}
\newtheorem*{definition*}{\bf Definition}
\newtheorem{note}[theorem]{}

\maketitle

\abstract{We investigate the vertex curve, that is the set of points in the hyperbolic region of a smooth
surface in real 3-space at which there is a circle in the tangent plane having at least 5-point contact
with the surface. The vertex curve is related to the differential geometry of planar sections of the surface
parallel to and close to the tangent planes, and to the symmetry sets of isophote curves, that is
level sets of intensity in a 2-dimensional image. We investigate also the relationship of the vertex curve
with the parabolic and flecnodal curves, and the evolution of the vertex curve 
in a generic 1-parameter family of smooth surfaces. }

\medskip
\noindent
MR Classification: 57R45, 53A05, 58K05
\section{Introduction}
This article is a contribution to the study of Euclidean invariants of surfaces,
and generic families of surfaces, in Euclidean space $\RR^3$. There have been
many previous such studies, involving among others
 contact of surfaces with spheres (ridge
curves, see for example~\cite{Porteous,Mumford,BGT3}),  right circular cylinders~\cite{Fukui}, and, as in the present article, circles. In \cite{Bruce} Bruce, following on from earlier work of
Montaldi~\cite{Montaldi2}, considers the contact of circles with surfaces, but the problems studied are different from ours. Another approach
is given in Porteous's book~\cite[Ch.15]{Porteous}.

In \cite{diatta-giblin} Diatta and the first author studied  vertices and inflexions of sections of a smooth surface $M$ in
$\RR^3$ by planes parallel to, and close to, the tangent plane $T_pM$ at a point $p$.  
This was in the context of families of curves which have a singular member (namely the section of $M$ by 
the tangent plane itself) and the behaviour of the symmetry sets of the curves in
such a family. (The corresponding evolution of symmetry sets of a 1-parameter family
of {\em smooth} plane curves was classified in \cite{BGsymmsets}.)
 This in turn was motivated by the fact that isophotes (lines of equal intensity) in a camera
image can be regarded as level sets of a function of position in the image, namely the intensity function. The evolution of vertices on planar sections parallel to $T_pM$ changes when $p$ crosses a certain curve on $M$, first studied in \cite{diatta-giblin}, and 
which we call the vertex curve (V-curve) in this article. {\em Thus the V-curve is a Euclidean invariant bifurcation set on $M$.}

\medskip

We first recall that the sign of the Gauss curvature $K$ distinguishes three types of points of a generic surface\,: 
{\em elliptic region} ($K>0$), {\em hyperbolic region} ($K<0$) and {\em parabolic curve} ($K=0$);  
and that a {\em vertex} of a plane curve $\gamma$ is a point where $\gamma$ has higher order of contact 
than usual with its osculating circle (at least $4$-point contact). 
\textit{At a vertex the radius of curvature of $\gamma$ is critical.}

\medskip
 
\noindent
\textbf{\small {\color{black}Vertex curve.}}
Let $M$ be a smooth surface in $\RR^3$.  The \textit{vertex curve}, or \textit{V-curve}, 
on $M$ is the closure of the set of points $p$ in the hyperbolic region of $M$ for which there 
exists a circle, lying in the tangent plane $T_p{\color{black}M}$ to $M$ at $p$, having (at least) 5-point 
contact with $M$ at $p$. Such a point $p$ is also called a \textit{vertex point}, or \textit{V-point}, of $M$. (The V-curves in this article were called `VT-sets' in~\cite{diatta-giblin}).
\medskip

At a hyperbolic point $p\in M$ the tangent plane $T_pM$ {\color{black}cuts the surface along two smooth 
transverse branches. Thus,} for a circle lying in the tangent plane
$T_pM$, 5-point contact
with $M$ at $p$ can be expected to mean that the circle meets one branch transversally (one of the five contacts)
and the other branch at a vertex
of that branch (the other four contacts); compare~\cite[\S3.4(1)]{diatta-giblin}.  
{\color{black}W}hence the name V-curve.

We study the structure of the V-curve and its interactions with the parabolic and flecnodal curves for a generic smooth surface of $R^3$ 
and investigate the changes which occur on V-curves during a generic 1-parameter deformation of the underlying surface.

The article is organized as follows.
In \S\ref{s:contact} we give two complementary methods for measuring the contact between a
surface $M$ at $p\in M$ and a circle  lying in the tangent plane to $M$ at $p$.  In \S\ref{s:vt-hyper}
we show that the V-curve is smooth on the hyperbolic region of $M$, and in \S\ref{ss:maxmin} we
show how to distinguish between vertices which are maxima or minima of the absolute radius of curvature.
In \S\ref{s:coG} we study the V-curve near a special parabolic point of $M$, namely a `cusp of Gauss' or `godron'
(defined below), showing in Proposition~\ref{prop:VatCoG} that at any `hyperbolic cusp of Gauss'
the V-curve has two smooth branches tangent to the parabolic curve (it is empty near an `elliptic cusp of Gauss'). In \S\ref{sss:projective} we introduce a Euclidean invariant of a cusp of Gauss, defined in two geometric ways.
In \S\ref{s:flecCoG} we find the interactions between the
V-curve and the flecnodal curve of $M$; the various possibilities are illustrated in Figure~\ref{fig:CoG1}.
In \S\ref{s:cod1} we investigate, partly experimentally, the evolution of the V-curve 
in a generic 1-parameter family of surfaces and finally in \S\ref{s:further} we mention some further ongoing
work.

\section{Contact function and contact map}\label{s:contact}
Here,  we describe two alternative ways to calculate the contact between a circle 
and the surface $M$. For the majority of this article we adopt the `standard calculation' 
below, in which we parametrize the circle and use a local equation $z=f(x,y)$ for $M$.
But for some purposes in \S\ref{s:cod1} we have found another useful 
approach in \S\ref{ss:contact2}: we parametrize the surface and use two equations
 for the circle.  The general theory of contact is contained in~\cite{Montaldi}.
 
\subsection{Computing the contact by the contact map}\label{ss:standard}

We assume $M$ to be locally given in Monge form $z=f(x,y)$, where $f$ and its partial 
derivatives $f_x, f_y$ vanish at the origin $(0,0)$, so that the tangent plane to $M$ 
at the origin is the coordinate plane $z=0$. In this method, we calculate the contact by 
composing a parametrization of the circle with the equation $z=f(x,y)$. 
Consider a circle or line through the origin in, say, the $(x_1, y_1)$-plane, given by
\begin{equation}
r(x_1^2+y_1^2)+sx_1+y_1=0,  \mbox{ with curvature } \frac{2r}{\sqrt{1+s^2}} \mbox{ and centre } \left( -\frac{s}{2r}, \ -\frac{1}{2r}\right), r\ne 0.
\label{eq:circlex1y1}
\end{equation}
The fact that this can represent a line ($r=0$) will be useful later.

We shall map this circle isometrically to a circle in the tangent plane $T_p$ at $p\in M$ by choosing an orthonormal
basis for $\mathbb R^3$ as follows. Write $p=(x_0, y_0,f(x_0,y_0))$ and let $f_x, f_y$ stand for the
partial derivatives of $f$ at $x=x_0, y=y_0$.
\begin{equation}
\ee_1=\frac{(1+f_y^2, -f_xf_y, f_x)}{||(1+f_y^2, -f_xf_y, f_x)||}, \ \ee_2=\frac{(0, 1, f_y)}{||(0, 1, f_y)||}, \ \ee_3=\frac{(-f_x, -f_y, 1)}{||(-f_x, -f_y, 1)||}.
\label{eq:e1e2e3}
\end{equation}
Thus $\ee_1, \ee_2$ span the tangent plane at $p$ and $\ee_3$ is a unit normal to $M$ at $p$. For $p=(0,0)$ the
three vectors form the standard basis for $\mathbb R^3$. We map a point
$(x_1,y_1)$ of the circle (\ref{eq:circlex1y1}) to
\[ (X,Y,Z)=(x_0, y_0, f(x_0,y_0))+x_1\ee_1+y_1\ee_2,\]
which lies on an arbitrary circle through $p$, lying in the tangent plane to $M$ at $p$.
When $x_0=y_0=0$ the map takes the circle in the $(x_1,y_1)$-plane
 identically to the
same circle in the $(x,y)$-plane which is the tangent plane to $M$ at the origin.

We shall parametrize the circle (\ref{eq:circlex1y1}) by $x_1$ close to the origin in the $(x_1,y_1)$ plane;
then the {\em contact function} between the corresponding circle in $T_p$ and the surface $M$ is
\begin{equation}
G(x_1,x_0,y_0,r,s) = Z-f(X,Y),
\label{eq:contact-hyper}
\end{equation}
where on the right-hand side $y_1$ is written as a function of $x_1$.  The vertex curve is the locus of points
$(x_0,y_0,f(x_0,y_0))$ for which the contact is at least five, and we shall need to find this curve close to the origin
$(x_0,y_0)=(0,0)$. The contact is at least five provided the first four derivatives of $G$ with respect
to $x_1$ vanish at $(0,x_0,y_0,r,s)$.

An important observation is that $x_1^2$ is a {\em factor} of the function $G$, 
that is $G(0,x_0,y_0,r,s)\equiv 0, \ G_{x_1}(0,x_0,y_0,r,s) \equiv 0$.
This is because the circle (\ref{eq:circlex1y1}) {\em always} has at least 2-point 
contact with the surface for $x_1=0$,
at the point $(x_0,y_0,f(x_0,y_0))$, since it passes through the intersection of 
the two curves in which the surface is met by
its tangent plane (or, at a parabolic point, through the corresponding singularity of the intersection).
\begin{defs} {\rm The smooth function $H$ determined by the equality 
\[ G(x_1,x_0,y_0,r,s) = x_1^2H(x_1,x_0,y_0,r,s)\]
will be called the} reduced contact function.
\label{def:H}
\end{defs}
We can now re-interpret the conditions that the first four derivatives of $G$ with respect to $x_1$ vanish
at $x_1=0$ in terms of the function $H$, as follows.
\begin{eqnarray*}
G_{x_1}&=&2x_1H+x_1^2H_{x_1} \\
G_{x_1x_1} &=& 2H+4x_1H_{x_1}+x_1^2H_{x_1x_1} \\
G_{3x_1}&=& 6H_{x_1}+6x_1H_{x_1x_1}+x_1^2H_{3x_1} \\
G_{4x_1}&=& 12H_{x_1x_1}+8x_1H_{3x_1}+x_1^2H_{4x_1}.
\end{eqnarray*}
Thus we now require $H= H_{x_1}=H_{x_1x_1}=0$ at $x_1=0$, that is we consider the map
\begin{eqnarray}
 \widetilde{H} &:& (\RR^4, 0)\to (\RR^3,0),  \nonumber\\
(x_0,y_0,r,s) &\mapsto & (H(0,x_0,y_0,r,s), H_{x_1}(0,x_0,y_0,r,s), H_{x_1x_1}(0,x_0,y_0,r,s)).
\label{eq:Htilde}
\end{eqnarray}
{\color{black}The projection to the $(x_0,y_0)$-plane of $\widetilde{H}^{-1}(0,0,0)$ is the set of points 
on $M$, near the origin, at which there is a circle on the tangent plane having $5$-point contact 
or higher with the surface. }

\subsection{An alternative approach}\label{ss:contact2}

Instead of parametrizing the circle and using an equation $z=f(x,y)$ for the surface $M$ we can parametrize the surface by
$(x,y)\mapsto (x,y,f(x,y))$ and write down two equations for the circle. We can specify a plane, namely the tangent plane to
$M$ at a given point $P_0=(x_0,y_0,f(x_0,y_0))$, and a sphere centred at a point of this plane and passing through $P_0$.
This gives a contact  map $\RR^2\to\RR^2$, which we can reduce using contact equivalence
($\mathcal K$-equivalence).   We shall use this method in \S\ref{s:cod1} as it makes the direct computations
much easier. 

With the notation above, let $(u,v,w)$ be a point in the tangent plane to $M$ at $P_0$. The equation of this
tangent plane is $G_1(x_0,y_0;x,y,z)=0$, given by the inner product
\[{\color{black}G_1=\langle (x-x_0,y-y_0,z-f(x_0,y_0))\,,\,(-f_x,-f_y,1)\rangle =0\,,}\]
and the partial derivatives are evaluated at $P_0$.  
Thus $w=(u-x_0)f_x+(v-y_0)f_y+f(x_0,y_0).$ The equation of the sphere centred at 
$(u,v,w)$ and passing through $P_0$ is $G_2(x_0,y_0,u,v;x,y,z)=0$ where 
$G_2=(x-u)^2+(y-v)^2+(z-w)^2-(x_0-u)^2-(y_0-v)^2-(z_0-w)^2$, $w$ being substituted as above.
The intersection of this sphere with
the tangent plane is the circle whose contact with $M$ at $P_0$ we wish to calculate.

To calculate the contact we
must parametrize $M$ close to $P_0$.  Thus let $(x_0+p, y_0+q)$ be parameters for $M$, where $p$ and
$q$ are small. The contact map, with variables $p,q$ and for fixed $x_0,y_0,u,v$, is then the composite of the parametrization
\[ (x_0,y_0;p,q) \mapsto (x,y,z)=(x_0+p,y_0+q,f(x_0+p,y_0+q))\]
with the map
\[ (x_0,y_0,u,v;x,y,z) \mapsto \left(G_1(x_0,y_0;x,y,z), \ G_2(x_0,y_0,u,v;x,y,z)\right).\]
We shall call the components of this composite map $(H_1(x_0,y_0,p,q), H_2(x_0,y_0,u,v,p,q))$.  Note that
when we use a polynomial approximation to $f$ both $H_1$ and $H_2$ are {\em polynomial} functions.

For fixed
$x_0,y_0,u,v$ it is a map (germ) $H:\mathbb R^2,0\to \mathbb R^2,0$ and its $\mathcal K$-class is an
alternative way of measuring the contact between a circle in the tangent plane to $M$ and the
surface $M$.  The parametrization of the V-curve consists of those 
$x_0,y_0$ for which, for some $u,v$, this map has
the contact type $A_4$ or higher at $p=q=0$.

\section{{\color{black}Vertex curve properties in the hyperbolic domain}}\label{s:vt-hyper}

We start with some basic background. 
A generic smooth surface $M$ of $\RR^3$ has three (possibly empty) parts\,:  
($H$) an open {\em domain of hyperbolic points}\,: at such points there are 
two tangent lines having greater than $2$-point contact with $M$, called {\em asymptotic lines};
($E$) an open {\em domain of elliptic points}\,: at such points there is no such line; \ 
 and  ($P$) a smooth {\em curve of parabolic points}\,: at such points there is a unique (double) asymptotic 
line. 

If $M$ is generic and locally given in Monge form $z=f(x,y)$ around $p$, then $p$ is hyperbolic, elliptic 
or parabolic if and only of the quadratic part of $f$,
called {\em second fundamental form of $M$ at $p$},
is respectively indefinite, definite or degenerate. 
The zeroes of this quadratic form (for $p$ hyperbolic or parabolic) are the asymptotic tangent lines at $p$. 
\smallskip

The integral curves of the fields of asymptotic tangent lines are called {\em asymptotic curves}. 
\medskip

\noindent 
\textbf{\small Left and right}. Fix an orientation in $\RR^3$. 
A regularly parametrized smooth space curve is said to be a \textit{left} (\textit{right})
\textit{curve} on an interval if its first three derivatives at each point form a negative 
(resp. positive) frame. Thus a left (right) curve has negative (resp. positive) torsion 
and twists like a left (resp. right) screw.
\medskip


\noindent
\textbf{Fact}. 
\textit{At each hyperbolic point $p$ one asymptotic curve is left and the other is right} (cf. \cite{ricardo}).
\medskip

The respective tangents $L_\ell$, $L_r$, called \textit{left} and \textit{right asymptotic lines}, 
are tangent to the smooth branches of the section $M\cap T_pM$. We call them {\em left} 
and {\em right} branches respectively. 
\medskip

\textit{This left-right distinction depends only on the orientation of $\RR^3$, but not of the surface.} 
\medskip

\noindent 
\textbf{\small Left and right vertex curve}. 
The \textit{left} (\textit{right}) vertex curve $V_\ell$ (resp. $V_r$) of a surface $M$ consists of
the points $p$ for which the left (resp. right) branch of $M\cap T_pM$ has a vertex.

\medskip\noindent
We study the behaviour of the V-curve close to the parabolic curve in \S\ref{s:coG}.
For more information on the behaviour of asymptotic curves close to the
parabolic curve see for example~\cite[Ch.\ 3]{BGM}, \cite[Ch.\ 6]{Shyuichi-et-al}.

{\footnotesize 
\begin{remark*}
At an elliptic point $p\in M$ the intersection with the tangent plane consists of two complex conjugate curves.  
In principle one can ask whether a complex circle in that tangent plane could have 5-point 
contact with $M$ at $p$. A calculation shows that this imposes two conditions on the point $p$, 
which implies that it is only possible at isolated points of a generic surface $M$.  
The two conditions imposed on the (real) coefficients in the Monge form of the surface do not appear 
to have any other geometrical meaning.
\end{remark*}
}

\subsection{Smoothness of the vertex curve at a hyperbolic point}\label{ss:hyper}
{\color{black}We shall take a surface in local Monge form at a hyperbolic point $p$ so that one 
asymptotic tangent line at $p$ is the $x$-axis $y=0$ and the other one is the line $x=ay$\,:} 
\begin{equation}
f(x,y)=xy-ay^2+b_0x^3+b_1x^2y+b_2xy^2+b_3y^3+c_0x^4+c_1x^3y+c_2x^2y^2+c_3xy^3+c_4y^4 + d_0x^5+\ldots,
\label{eq:hyperbolic}
\end{equation}
where, if they are needed, the degree $5$ terms will have coefficients $d_0, \ldots, d_5$, and so on.

In what follows, we shall consider the asymptotic direction along the $x$-axis.

\begin{prop}
{\color{black}In a generic smooth surface $M$, each branch (the left $V_\ell$ and the right $V_r$) of the 
$V$-curve is nonsingular on the hyperbolic domain.}  
 \label{prop:singV}
\end{prop}

\begin{proof}
Applying \S\ref{ss:standard} to $x_0=y_0=0$, we clearly need $s=0$ for the circle to be tangent to the
branch of $f=0$ tangent to the $x$-axis (for $f$ given in \eqref{eq:hyperbolic}). The contact function then becomes
\[ (r-b_0)x_1^3+(ar^2+b_1r-c_0)x_1^4+(r^3-b_2r^2+c_1r-d_0)x_1^5+ \ \mbox{higher terms},\]
so that for 4-point contact we need $r=b_0$ and for exactly 5-point contact we add
\begin{equation}
ab_0^2+b_1b_0-c_0=0, \ {\color{black}\mbox{ and }A\neq 0 \ \mbox{with } A:=b_0^3-b_2b_0^2+c_1b_0-d_0}\,.
\label{eq:Vatorigin}
\end{equation}
The {\color{black}condition $r=b_0$ ensures that} the circle osculates the branch of $f=0$ 
tangent to the $x$-axis, that is $b_0(x^2+y^2)+y=0$ is the equation, in $(x,y)$-coordinates 
in the plane $z=0$, of the osculating circle of this branch, with centre $(0,-\frac{1}{2b_0})$ 
and curvature $2b_0$.
{\color{black}The additional condition $ab_0^2+b_1b_0-c_0=0$ ensures} that the origin is a V-point, while  
{\color{black}the condition $A\neq 0$, with $A$ as given in $(6)$, ensures that the corresponding circle has
exactly $5$-point contact with the surface}.  

Referring to (\ref{eq:Htilde}), we need to study $\widetilde{H}^{-1}(0,0,0)$ and its 
Jacobian matrix at $(0,0,r_0,s_0)$ for suitable values of $r_0$ and $s_0$, that is for 
values which correspond to those for a circle which {\em does} have 5-point contact with 
the surface at the origin.  Of course this requires the origin on the surface $M$ to be 
a vertex point. As we shall see, for a hyperbolic point on the surface this gives a single 
condition on the point, meaning that vertex points generically lie on a curve on the surface. 
Thus for a smooth vertex curve---the locus of vertex points---at $p$ we require that
\begin{enumerate}
\item $\widetilde{H}(0,0,r_0,s_0)={\color{black}(0,0,0)}$ for some $r_0,s_0$; this is the same as (\ref{eq:Vatorigin}) above, that is
$r_0=b_0, s_0=0$ and $ab_0^2+b_1b_0-c_0=0$,
\item the $3\times 4$ Jacobian matrix of $\widetilde{H}$ at $(0,0,r_0,s_0)$ has rank 3, and
\item the third and fourth columns of the Jacobian matrix are independent.
\end{enumerate}
The second condition ensures that $\widetilde{H}^{-1}(0,0,0)$ is smooth at $(0,0,r_0,s_0)$ 
and the third condition ensures that the projection of this set to the $(x_0,y_0)$-plane is 
also smooth at $p=(0,0)$.

From now on in this section we assume condition (\ref{eq:Vatorigin}) on $c_0$.
The Jacobian matrix $J$ of $\widetilde{H}$ at $(0,0,b_0,0)$ takes the form (from a direct calculation)
\[
J=\left( \begin{array}{cccc}
-3b_0 & -b_1& 0 & 1 \\
-4ab_0^2-2b_0b_1 & 2b_0b_2-c_1 & 1 & 2ab_0+b_1 \\
-2b_0^2b_2+6b_0c_1-10d_0 &-6b_0^2b_3+4b_0c_2-2d_1 & 4ab_0+2b_1 & 4b_0^2-4b_0b_2+2c_1
\end{array}
\right).
\]
The last two columns of $J$ are always independent, so that provided one of the minors consisting of columns
1,3,4 or 2,3,4 is nonzero, the whole matrix has rank 3 and the vertex curve is
smooth in a neighbourhood of our point $p$. Putting both these minors equal to
zero gives formulas for $d_0$ and $d_1$ in terms of $b_0,b_1,b_2,c_1,c_2$, bearing in mind that $c_0=ab_0^2+b_0b_1$.
This imposes two additional conditions on the point, and hence does not occur on a generic surface. 
\end{proof}

A generic surface may have isolated points at which the circle has higher contact\,:
\medskip

\noindent
{\color{black}\textbf{\small Bi-vertex}. A point of the surface where a circle in the tangent plane has $6$-point 
contact with $M$, that is where one branch of the curve $M\cap T_pM$ has a degenerate vertex, 
is called a \textit{bi-vertex}.}

\begin{rems}\label{condition-bivertex}
{\color{black}
(1) On a generic surface, the condition $A\neq 0$ holds along the $V$-curve, except at the bi-vertices.
Since the equality $A=0$ (implying $6$-point contact) does not affect the proof for the $V$-curve to be smooth 
(Proposition~\ref{prop:singV}), the $V$-curve is still smooth at a bi-vertex. }

\noindent 
(2) The above proof shows that the tangent vector to the vertex curve at $p$ depends on the terms 
$b_0, b_1, b_2, c_1,c_2, d_0$ and $d_1$. This tangent vector comes to
\begin{eqnarray}
&& (4a^2b_0^2b_1+4ab_0^2b_2+4ab_0b_1^2-2ab_0c_1-2b_0^2b_1+3b_0^2b_3+4b_0b_1b_2+b_1^3-2b_0c_2-2b_1c_1+d_1,
 \nonumber \\
&& -4a^2b_0^3-4ab_0^2b_1+6b_0^3-7b_0^2b_2+b_0b_1^2+6b_0c_1-5d_0).
\label{eq:singV}
\end{eqnarray}
\noindent
(3) If the second component of  {\rm (\ref{eq:singV})} is $0$ and the first is nonzero then
 the V-curve is tangent to the corresponding branch of  the intersection of $M$ with its tangent
 plane at the origin.
\end{rems} 

{\color{black}A generic surface may have also the following isolated points\,:
\medskip

\noindent
\textbf{\small Vertex-crossing} or \textbf{\small V-crossing}. 
A point of transverse intersection of $V_\ell$ and $V_r$, the left and 
right (smooth) branches of the $V$-curve, is called \textit{vertex-crossing} or \textit{V-crossing}. 
\smallskip

Thus, at a vertex-crossing each of the two smooth curves comprising 
$M\cap T_pM$ has a vertex. }

Assuming $c_0=ab_0^2+b_0b_1$ for $M$ locally given in Monge form $(5)$ so that the branch of 
$M\cap T_pM$ tangent to the $x$-axis has a vertex, the additional condition for the branch tangent 
to $x=ay$ to have a vertex is the following, which can be regarded as a condition on $c_4$\,: 
\begin{eqnarray*}
b_2b_3-c_4+(2b_1b_3+b_2^2-2b_3^2-c_3)a+(3b_0b_3+3b_1b_2-3b_2b_3-c_2-c_4)a^2+&&\\
(4b_0b_2+2b_1^2-2b_1b_3-b_2^2-c_1-c_3)a^3+(4b_0b_1-b_0b_3-b_1b_2-c_2)a^4+&& \\
(2b_0^2-c_1)a^5=0.&&
\end{eqnarray*}
The osculating circle of this branch at the origin is of the form (\ref{eq:circlex1y1}) with
 \[r=-\frac{a^3b_0+a^2b_1+ab_2+b_3}{a(a^2+1)} \mbox{ and }  s=-\frac{1}{a},\]
 provided $a\ne 0$. (The form (\ref{eq:circlex1y1}) is not adapted to circles 
 whose centre is on the $x$-axis.)

\subsection{Maximum and minimum points}\label{ss:maxmin}

We now seek to distinguish between {\em maximum} and {\em minimum} points. This means:
consider the intersection $X=M\cap T_pM$ at a hyperbolic point, where $p$ belongs to the V-curve. 
Then one branch of $X$, say $X_{\color{black}\ell}$, has a vertex at $p$. Does this vertex correspond 
to a maximum or a minimum of the (absolute) radius of the osculating circle at points of 
$X_{\color{black}\ell}$? 

{\color{black}
\begin{prop}\label{prop:maxmin-hyper}
{\color{black}Let $p\in M$ be a hyperbolic point of the $V$-curve, and take $A$ as in \eqref{eq:Vatorigin}.

\noindent
$(a)$ The absolute radius of curvature $|\kappa^{-1}|$ of the corresponding branch 
of $M\cap T_pM$ has a minimum (maximum) if and only if $rA>0$ $($resp. $rA<0$$)$. 
\smallskip

\noindent
$(b)$ The corresponding branch of $M\cap T_pM$ has a degenerate (double) vertex if and only if $A=0$. 
In this case, $p$ is a bi-vertex and locally separates the $V$-curve into a half-branch of maxima and 
a half-branch of minima. } {\rm (See Figure\,\ref{fig:relevant-v-points}, left.)}
\end{prop}
}

\begin{proof} $(a)$ The curvature of the component of $X$ which is tangent to the $x$-axis comes to
\begin{equation}
 \kappa = -2b_0+12Ax^2+ \ldots \mbox{ where } A=b_0^3-b_0^2b_2+b_0c_1-d_0, \mbox{ as in } (\ref{eq:Vatorigin}).
 \label{eq:A}
 \end{equation}
{\color{black}This implies the absolute radius of curvature $|\kappa^{-1}|$ has a minimum (maximum) at $x=0$ 
if and only if $b_0A>0$ (resp. $b_0A<0$).}
Further, as above (\ref{eq:Vatorigin}), $b_0$ is the value of $r$ at $p$. To find $A$ we consider the
next  derivative of the reduced contact function $H$ (Definition~\ref{def:H}).  We get 
\[ \frac{\partial^3H}{\partial x_1^3}(0,0,0,b_0,0)=6A\,,\]
which is zero if and only if the branch of the intersection $M\cap T_pM$ tangent to the $x$-axis has a degenerate vertex\,:
a circle in the tangent plane has 6-point contact with the surface {\color{black}(a bi-vertex; see
Remarks\,\ref{condition-bivertex}(1))}.  

\noindent
{\color{black}$(b)$ In a generic surface, the function $A$ has only simple zeroes on the V-curve 
(at the bi-vertices). Thus a bi-vertex $p$ locally separates the V-curve into two half-branches\,: 
in one branch $A>0$ and in the other $A<0$. 
Item $(b)$ follows from item $(a)$ because $r$ does not change sign at $p$.}
\end{proof}

\subsection{Flecnodal curve and biflecnodes}\label{ss:flec}

\textbf{\small Flecnodal curve}. In the closure of the hyperbolic domain of $M$ there is a smooth 
immersed {\em flecnodal curve} $F$ formed by the points satisfying 
any of the equivalent conditions (F1)-(F4):
\begin{enumerate}
\item[(F1)]  An asymptotic line (left or right) exceeds $3$-point contact with $M$ at $p$.
\item[(F2)]   An asymptotic curve through $p\in M$ (left or right) has an inflexion---that
is, for a regular parametrization the first two derivatives are dependent (proportional) vectors. 
\item[(F3)]  A smooth branch (left or right) of the tangent section $M\cap T_pM$ at $p$ has an inflexion.
\item[(F4)]  In terms of the Monge form (\ref{eq:hyperbolic}), $b_0=0$ and $c_0\ne 0$.
\end{enumerate}

To see why (F4) is equivalent to both (F2) and (F3), note that  the asymptotic curve through
$p$, corresponding to the asymptotic direction $y=0$ in (\ref{eq:hyperbolic}), has expansion (as a space curve)
\[x\mapsto \left(\,x, \,-\textstyle{\frac{3}{2}}\displaystyle b_0 x^2 + \ha(6a b_0^2+5b_0 b_1 - 4c_0)x^3 + \ldots,\,0\,\right)\]
and the corresponding branch of the plane curve $M\cap T_pM$ has expansion
\[x\mapsto \left(\,x, \,-b_0x^2+(ab_0^2+b_0b_1-c_0)x^3+\ldots\,\right).\]

\noindent
\textbf{\small Left and Right Flecnodal Curve}. 
The {\em left} ({\em right}) {\em flecnodal curve} $F_\ell$ (resp. $F_r$) of $M$ 
consists of the points of $F$ at which the over-osculating asymptotic line
is of left (resp. right) type.  
\medskip

A generic surface may have isolated points of transverse 
intersection of the left and right branches of the flecnodal curve, called {\em hyperbonodes}.
The presence of hyperbonodes is necessary for the metamorphosis of the parabolic curve 
in generic $1$-parameter families of surfaces \cite{Uribeevolution}. 
A detailed study on the geometry of hyperbonodes 
was done in \cite{Uribeinvariant}, \cite{Kazarian-Uribe}. 
We can also find isolated points of the flecnodal curve at which the asymptotic line 
exceeds $4$-point contact\,:
\medskip

{\color{black}
\noindent
\textbf{\small Biflecnode}. 
A point at which a line has 5-point contact with the surface is called \textit{biflecnode}. 
\medskip

Hence a biflecnode is a V-point with $r=0$ (a circle of infinite radius). 
Therefore \textit{a biflecnode is a point of transverse intersection of the left (or right) branches 
of the flecnodal and vertex curves}.}  
At a {\em biflecnode} both the asymptotic curve and the intersection 
curve $M\cap T_pM$ have a {\em second order} inflexion.

{\color{black}
\begin{rem}
We obviously get a biflecnode from (\ref{eq:hyperbolic}) by taking $b_0=c_0=0, \ d_0\ne0$.
\label{prop:biflecnode-conditions}
\end{rem}

\begin{prop}
A left (right) biflecnode locally separates the left (resp. right) $V$-curve into a half-branch 
of maxima and a half-branch of minima. 
\label{prop:biflec-maxmin}
\end{prop}
\begin{proof}
The statement follows from Proposition~\ref{prop:maxmin-hyper}\,$(a)$ because at a biflecnode $p$ 
of a generic surface we have $A\neq 0$ and the value of $r$ (i.e., of $b_0$) changes sign at $p$ 
(cf. Remark\,\ref{prop:biflecnode-conditions}). 
\end{proof}

\subsection{Stable isolated vertex points in the hyperbolic domain}
Some of the different possibilities for the above isolated vertex points are shown in 
Figure\,\ref{fig:relevant-v-points} (see Proposition~\ref{prop:maxmin-hyper} and \ref{prop:biflec-maxmin}). 
A bi-vertex may be left or right; there are four types of V-crossings (the branches $V_\ell$ and 
$V_r$ may consists of maxima or of minima); a biflecnode may be left or right. }
\begin{figure}[h]
\centerline{\includegraphics[width=4.6in]{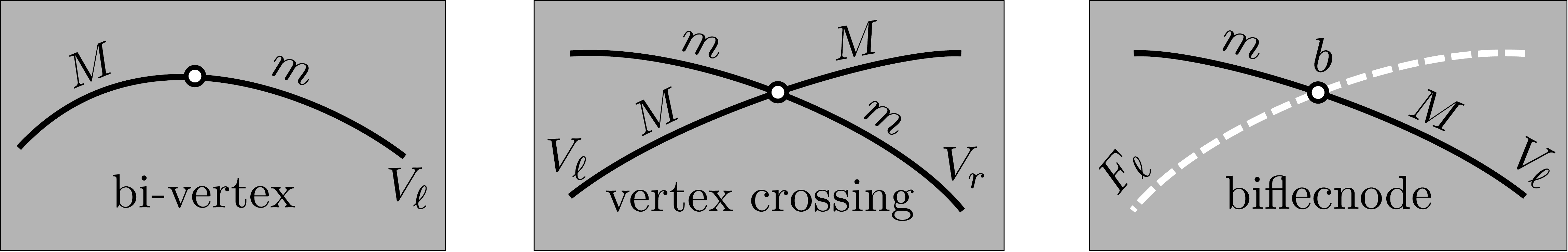}}
\caption{\small A left bi-vertex, a vertex crossing (with $V_\ell$-Max, $V_r$-min) and a left 
biflecnode.}
\label{fig:relevant-v-points}
\end{figure}

A corollary of Proposition~\ref{prop:maxmin-hyper}\,$(b)$ and Proposition~\ref{prop:biflec-maxmin} 
is the 
\begin{theorem}
  On each connected component of the V-curve of a compact generic surface in $\RR^3$ the number 
  of bi-vertices plus the number of biflecnodes is even. \\
  On each connected component of the V-curve of a compact generic orientable surface of $\RR^3$ 
  there is an even number {\rm (possibly $0$)} of bi-vertices and an even number {\rm (may be $0$)} 
  of biflecnodes. 
\end{theorem}

\section{Vertex curves at a cusp of Gauss}\label{s:coG}

\medskip\noindent
One of the most remarkable  points of a generic surface $M$ is a
\smallskip 

\noindent
\textbf{\small Cusp of Gauss}. Assume that the parabolic curve of $M$ is smooth.
A {\em cusp of Gauss} is a parabolic point at which the unique (but double) asymptotic  line is tangent to the parabolic curve. 

\medskip
\noindent{\small{\bf Note on terminology}} Two other common names for a cusp of Gauss
(that is, a cusp of the Gauss map)
are `godron', favoured by Ren\'{e} Thom, and `ruffle', used in J.Koenderink's well-known book
\cite{solidshape}. These names have the advantage that they do not suggest a Euclidean setting, and
indeed the cusp of Gauss is actually a {\em projectively} invariant concept; see \S\ref{sss:projective} below.
This article is about Euclidean concepts so we shall stick to `cusp of Gauss', except in circumstances
where this would prove unwieldy, as in `flecgodron' (\S\ref{ss:flecgodron}).

\subsection{Some basic properties of cusps of Gauss}
Cusps of Gauss have lots of interesting properties. Let us mention two of them\,: 
\medskip

\noindent
 \textit{All curves on $M$ tangent to the parabolic curve at a cusp of  Gauss 
$g$ have torsion zero at $g$},
 \cite{ricardo}. 
 \smallskip
 
Therefore the space of $2$-jets of such tangent curves, $J_g^2:=\{(t,\frac{1}{2}ct^2,0):c\in\RR\}\approx \RR$, 
is identified (up to a factor $\frac{1}{2}$) with the set of their curvatures $\{c\in\RR\}$.

\begin{slemma}[\cite{ricardo}]
Given a cusp of Gauss $g\in M$, there exists a unique $2$-jet $($curvature$)$ $\sigma$ in $J_g^2$ {\rm (called 
{\em separating $2$-jet at $g$})} satisfying the following properties\,$:$ 
\smallskip

\noindent
$(a)$ The images, by the Gauss map $\Gamma:M\rightarrow\mathbb{S}^2\subset\RR^3$, of all curves of 
$M$ tangent to the asymptotic line at $g$ and whose curvature at $g$ is different from $\sigma$ 
are semi-cubic cusps of $\mathbb{S}^2$ sharing the same tangent line at $\Gamma(g)$. 
\medskip

\noindent
$(b)$ {\sl Separating property}\,$:$ The images under $\Gamma$ of any two curves tangent to the asymptotic 
line at $g$, whose $2$-jets $($curvatures$)$ are separated by $\sigma$, are cusps pointing in opposite directions.  
\end{slemma}

\noindent
\textbf{\small Separating invariant.} The number $\sigma$ given in the above lemma is a Euclidean invariant 
of the cusp of Gauss $g\in M$ that we call the \textit{separating invariant}. 
\medskip

\noindent 
\textbf{\small Monge form}. 
Let $p\in M$ be a parabolic point. We shall take $p$ as the origin and the asymptotic line 
at $p$ as the $x$-axis. Then the (degenerate) quadratic part of the Monge form is $y^2$\,: 
\begin{equation}
 z = y^2+b_0x^3+b_1x^2y+b_2xy^2+b_3y^3 + c_0x^4+ c_1x^3y+ c_2x^2y^2+c_3xy^3+c_4xy^4+d_0x^{\color{black}5}+ \ldots.
 \label{eq:coG}
\end{equation}
\begin{lemma}\label{lemma:Monge godron}
Assume as before that the parabolic curve is smooth.
A parabolic point of a  surface in Monge form\,\eqref{eq:coG} is a cusp
of Gauss if and only 
if $b_0=0$ and $b_1\neq 0$. {\rm (We will see below that $-b_1$ is a Euclidean invariant of cusps of Gauss.)} 
\end{lemma}
\begin{proof}
The local equation of the parabolic curve $P$, \, $f_{xx}f_{yy}-f_{xy}^2=0$,\, 
starts with the terms 
\[ 3b_0x+b_1y+\ldots=0\,. \]
Thus the asymptotic line at $p$ ($y=0$) is tangent to the parabolic curve at $p$ 
if and only if $b_0=0$ and $b_1\neq 0$ (because the parabolic curve is smooth). 
\end{proof}


\begin{coGspecial}\label{item:specialcoG}
{\rm \textbf{\small Simple and special cusps of Gauss} 
(a) The condition for the image of the Gauss map at the origin to be an
ordinary (semi-cubical) cusp, using the above form (\ref{eq:coG}), is
$b_1^2-4c_0\ne 0$. When this holds, we say the cusp of Gauss is
{\em simple} (sometimes called {\em nondegenerate}). {\em On a generic surface all
cusps of Gauss are simple.}

\noindent
(b) The condition $b_1^2-4c_0\ne 0$  is also the condition for the
height function $z(x,y)$ in the normal direction $(0,0,1)$ at the origin, that is the
contact function between $M$ and its tangent plane at the origin, to have type
exactly $A_3$.

\smallskip\noindent
(c) The height function can degenerate in two ways: to type $A_4$ or to $D_4$.
 Both these are non-generic for a single surface but occur generically in
 1-parameter families; we explore such families in \S\ref{s:cod1}.
 
 \smallskip\noindent
 In the case of $A_4$, also called a double 
 cusp of Gauss, or {\em bigodron}, the parabolic curve remains smooth
 ($b_1\ne 0$), and 
 $b_1^2-4c_0=0, \ b_1^2b_2-2b_1c_1+4d_0\ne 0$. 
 This can be regarded as the collapse of two simple cusps of Gauss, one elliptic
 and one hyperbolic. See \S\ref{ss:degen-coG}. (This is also sometimes
 called a {\em degenerate} cusp of Gauss but the term is ambiguous and `double'
 is a more descriptive term.)
 
 \smallskip\noindent
 In the case of $D_4$, 
 also called a {\em flat umbilic},
 the parabolic curve becomes singular. See \S\ref{ss:D4}.
 
 \smallskip\noindent
 (d) There is also the possibility that the parabolic curve undergoes a `Morse transition',
 becoming singular at the moment of transition. See \S\ref{ss:Morse}.
 }
\end{coGspecial}

\bigskip
\noindent
We now show the following.
\textit{The only common points of the vertex curve and the parabolic curve are cusps of Gauss\,}:
\smallskip

\noindent
\textbf{Proposition}. 
\textit{If a parabolic point $p$ of a generic surface is a vertex point, then $p$ is a cusp of Gauss.} 

\begin{proof}
Let $p$ be a parabolic point of $M$, for $M$ is locally given in Monge form\,\eqref{eq:coG}. 
If $p$ is also a vertex point, it is easy to check that the contact function 
(\ref{eq:contact-hyper}) at $p=(0,0)$ takes the form
\[ b_0x_1^3+(r^2-b_1r+c_0)x_1^4+\ldots.\]
Referring to the circle given by (\ref{eq:circlex1y1}), we must have
$s=0$ to ensure that the 5-point contact circle is tangent to the intersection curve $M\cap T_pM$. 
This implies the 
\begin{lemma}\label{prop:rs}
There is 5-point contact at $(x_0,y_0)=(0,0)$ if and only if $b_0=0$, $r$ is a real solution of
$r^2-b_1r+c_0=0$, and $s=0$.  The  curvature of this circle is $2r$. 
\end{lemma}
\noindent 
Lemma\,\ref{lemma:Monge godron} and Lemma\,\ref{prop:rs} imply that $p$ is a cusp of Gauss for which $b_1^2-4c_0>0$. 
\end{proof} 

\begin{defs}\label{def:coG}
A cusp of Gauss  is said to be {\em hyperbolic} if the intersection with the tangent plane is 
two tangential curves, that is $b_1^2-4c_0>0$. 
A cusp of Gauss is said to be {\em elliptic} if the intersection with the tangent plane is 
an isolated point, that is $b_1^2-4c_0<0$. \\
(In \cite{ricardo} there are five other geometric characterisations of elliptic and hyperbolic cusps of Gauss.)
\end{defs}

\noindent 
\textit{A cusp of Gauss belongs to the vertex curve if and only if it is hyperbolic.} (By Lemma~\ref{prop:rs}.)  

\medskip
\noindent
\textbf{\small Remark} \
At a hyperbolic cusp of Gauss neither of the two tangential curves comprising $M\cap T_pM$  
has a vertex at $p$, {\color{black}but their respective osculating circles have $5$-point contact 
with the surface ($3$-point contact with the osculating branch and $2$-point contact with the other 
tangent branch).}

\subsubsection{Projective and Euclidean invariants of cusps of Gauss}\label{sss:projective}
In fact cusps of Gauss are projectively invariant. 
Platonova's (projective) normal form of the $4$-jet of a surface at a cusp of Gauss
 $g$ is 
$z=\frac{1}{2}y^2-x^2y+\frac{1}{2}\rho x^4$, where $\rho$ is a projective invariant 
defined in \cite{ricardo} as a cross ratio. 
A cusp of Gauss $g$ is hyperbolic (resp.\ elliptic) if and only if $\rho<1$ 
(resp. $\rho>1$), and simple if and only if $\rho\ne 1$.
Computing the cross-ratio invariant $\rho$ in Monge form\,\eqref{eq:coG} (with $b_0=0$), 
we get $\rho=4c_0/b_1^2$.

In our Euclidean case, other coefficients of \eqref{eq:coG} will also play a role. For example, 

\begin{prop*}{\rm (Uribe-Vargas, unpublished.)}\label{prop:curvature-separating}
At a cusp of Gauss, the curvature of the line of $($zero$)$ principal curvature is equal to the 
separating invariant $\sigma$. In Monge form\,\eqref{eq:coG}, with $b_0=0$, this curvature is equal to $-b_1$. 
\end{prop*}

Then the coefficient $-b_1$ represents the geometric and purely Euclidean invariant $\sigma$. Thus we shall write
\begin{equation}
\label{eq:sigma-rho}
 b_1=-\sigma, \ \ c_0=\textstyle{\frac{1}{4}}\displaystyle \sigma^2\rho. 
 \end{equation}

\subsection{Tangency of the parabolic and vertex curves at a cusp of Gauss}
\textbf{\small `Naturally' oriented coordinates}. At each elliptic point $p$ the surface lies 
locally on one of the two half-spaces determined by its tangent plane at $p$, called the
\textit{positive half-space at $p$}. By continuity, the positive half-space is well defined 
at parabolic points. At a cusp of Gauss $g$, direct the positive $z$-axis to the positive half-space at $g$, 
the positive $y$-axis towards the hyperbolic domain, and the positive $x$-axis in such way that any basis 
$(e_x, e_y, e_z)$ of $x, y, z$ forms a positive frame of the oriented $\RR^3$. 

Using the local Monge form of $M$ at a cusp of Gauss (see (\ref{eq:sigma-rho})) 
\begin{equation}
z=f(x,y)=y^2-\sigma x^2y+b_2xy^2+b_3y^3+\textstyle{\frac{1}{4}}\displaystyle \sigma^2\rho x^4 
+c_1x^3y+c_2x^2y^2+c_3xy^3+c_4y^4+\ldots
\label{eq:hypCoG}
\end{equation}
we find that, when $b_1<0$, the elliptic domain is on the side $y<0$ of the tangent line
$y=z=0$ to the parabolic curve at the origin and the positive $z$-axis is 
the limit of normals to $M$ directed into the positive
half-space supporting $M$ at these elliptic points. Therefore the $x,y,z$ axes
are naturally oriented as above.

 \smallskip
{\em We therefore assume $b_1 =-\sigma < 0$ from now on.}
\smallskip

\begin{prop}
Let $g$ be a simple (\S\ref{item:specialcoG})
hyperbolic cusp of Gauss  of a generic smooth surface $M$.   
In a neighbourhood of $g$, the $V$-curve consists of two smooth curves, 
tangent to the parabolic curve at $g$, and having at least
$3$-point contact with each other. 
For $M$ locally given in Monge form \eqref{eq:hypCoG} the  condition 
for exactly 3-point contact is $c_1+\sigma b_2\neq 0$.
\label{prop:VatCoG}
\end{prop}

\begin{proof}
At a hyperbolic cusp of Gauss $g$ there are two distinct circles having 5-point contact 
with the surface at $g$ (Lemma\,\ref{prop:rs}). 
Thus there are two branches of the vertex curve through the cusp of Gauss $g$.  
We shall prove that these branches are smooth and tangential there.\footnote{
In \cite[p.86]{diatta-giblin} it is stated that  a V-curve does not always exist in a 
neighbourhood of a hyperbolic cusp of Gauss.  This is incorrect.}
\medskip

Following the method of \S\ref{s:contact}, we evaluate the Jacobian matrix $J$ of the map 
$\widetilde{H}$ at $(0,0,r_0,0)$ (see  \eqref{eq:Htilde}), where 
$r_0=\frac{1}{2}(-\sigma+\sqrt{\sigma^2-4c_0})=-\frac{1}{2}\sigma(1-\sqrt{1-\rho})$ 
is one of the two  values of $r$, we obtain a matrix whose third and fourth columns are 
\[ \left(0,0,-2\sigma\sqrt{1-\rho}\right)^\top \mbox{ and } \left(0,-\sigma\sqrt{1-\rho},2c_1-4b_2r_0\right)^\top,\]
which are  independent since $\rho\ne 1$ for a simple cusp of Gauss.
The $3\times 3$ minors formed by columns 1,3,4 and 2,3,4 are respectively
0 and $-2\sigma^3(1-\rho)$, therefore the branch of $\widetilde{H}^{-1}(0)$ and the
corresponding branch of the vertex curve of the surface $M$ at the origin are smooth and
can both be parametrized locally by $x_0$, provided {\color{black}the cusp of Gauss is simple.}

The first row of the Jacobian matrix is $(0,\sigma,0,0)$ and this implies that (given $\sigma\ne 0$) 
a kernel vector of this matrix has the form $(\xi_1, 0, \xi_3, \xi_4)$ for some $\xi_1, \xi_3, \xi_4$ 
where $\xi_1\ne 0$ since the projection of the tangent vector to the first two coordinates is not zero.  
Hence the tangent to this local branch of the vertex curve at $g$ is $(1,0)$ 
in the $(x,y)$-plane, or $(1,0,0)$ in the ambient 3-space.  

The same applies to the other local branch of the vertex curve, and therefore {\color{black}both} local
branches are tangent {\color{black}to the parabolic curve at $g$.}

Applying the same method as \S\ref{s:vt-hyper} to \eqref{eq:hypCoG}, we find the initial terms of the parametrization 
of the two local branches of the vertex curve 
\begin{equation}
V^1: y=\ha\sigma\rho x^2 + B_1x^3+\ldots, \qquad V^2: y =\ha\sigma\rho x^2 + B_2x^3+\ldots,
\label{eq:vt-CoG}
\end{equation}
where $B_1-B_2=8\sqrt{1-\rho}\,(\sigma b_2+c_1)$. 
Thus provided $c_1+\sigma b_2\neq 0$,
the two local branches have exactly 3-point contact, and therefore will cross
tangentially at $g$.  
\end{proof}

\begin{rem}\label{left-right-orient-F}
{\color{black}
It is well known that at every cusp of Gauss of a generic smooth surface 
the flecnodal curve $F$ is also tangent to the parabolic curve $P$. 
Moreover, \textit{cusps of Gauss locally separate the flecnodal curve into 
left and right half-branches, and the local right-to-left orientation of $F$, at a hyperbolic cusp of Gauss, 
coincides with the negative-to-positive orientation of the $x$-axis 
in our oriented coordinates} \cite{ricardo}.} 
For the V-curve we have a similar statement\,: 
\end{rem}

\begin{prop}\label{prop:left-righ_Vcurve-at-g}
At a hyperbolic cusp of Gauss $g$, each tangential component $V^1$, $V^2$ of the {\rm V}-curve is 
locally separated by $g$ into left and right half-branches. The local right-to-left orientation 
of the curve $V^2$ coincides with the negative-to-positive orientation of the $x$-axis 
{\rm (like the flecnodal curve $F$)} and is opposite to that of $V^1$.  
\end{prop}
\begin{proof}
For a point $p$ of the V-curve close to $g$ we shall write the asymptotic 
directions on $M$ at $p$, projected to the plane $z=0$, as $(1,P)$. One can easily verify that 
the two asymptotic directions (at the hyperbolic points near $g$), projected to the 
$(x,y)$-plane, satisfy that the slope of the left asymptotic line is $<$ 
the slope of the right one. 
 
For a point $(x_0,y_0(x_0), f(x_0, y_0))$ of $V^1$ close to $g$
we find that to first order in $x_0$ the asymptotic direction tangent to the branch 
of $M\cap T_pM$, having a vertex, is $(1, \sigma(1+\sqrt{1-\rho})x_0,0)$, and  
the respective asymptotic direction for $(x_0,y_0(x_0), f(x_0, y_0))$ of $V^2$ is 
$(1,\sigma(1-\sqrt{1-\rho})x_0,0)$.

For a parabolic point near $g$ the unique asymptotic direction, to first order in $x_0$, is 
$(1,\sigma x_0, 0)$. 

Then for the hyperbolic points, with fixed $x=x_0$, near $g$ the slope $P_\ell$ 
of their left asymptotic line must satisfy $P_\ell<\sigma x_0$. the condition for the right 
asymptotic lines is $P_r>\sigma x_0$. 

Thus for points $p=(x_0,y_0,f(x_0,y_0))$ of the tangential component $V^1$ or $V^2$ of the V-curve close to $g$ we have 
\[ V^1 \ \mbox{is right at } p \ \Longleftrightarrow \sigma(1+\sqrt{1-\rho})x_0 > \sigma x_0 \Longleftrightarrow x_0 > 0\,,\]
\[ V^2 \ \mbox{is left at } p \ \Longleftrightarrow \sigma(1-\sqrt{1-\rho})x_0 < \sigma x_0 \Longleftrightarrow x_0 > 0.\]
Therefore Proposition~\ref{prop:left-righ_Vcurve-at-g} is proved.
\end{proof}


\section{Further interactions at cusps of Gauss}\label{s:flecCoG}

\subsection{Configurations of geometrically defined curves at cusps of Gauss}
Write $T_-$ and $T_+$ for the two branches of the tangent section $M\cap T_gM$ at $g$. 
We shall determine the relative positions (near $g$) of the vertex curves $V^1, V^2$, 
the flecnodal curve $F$, the parabolic curve $P$, the branches $T_\pm$ of $M\cap T_{g}M$, 
and the line $C$ of (zero) principal curvature through $g$. 
Since all these curves are tangent to the asymptotic line at $p$, 
their $2$-jet is a curve on the tangent plane of the form $y=\frac{1}{2}cx^2+\ldots$. 
Therefore the local configurations of $F$, $P$, $V$, $T_\pm$, $C$ and the asymptotic line at $g$ 
are determined by the relative positions of their respective curvatures 
$c_{\sF}$, $c_{\sP}$, $c_{\sV}$, $c_{\sT_-}$, $c_{\sT_+}$ and $c_{\sC}=\sigma$ on the real line. 

\begin{theorem}\label{rho-classification_tangent}
Given a simple hyperbolic cusp of Gauss $g$ of $M$, there are seven possible configurations of the 
curves $F$, $P$, $V$, $T_\pm$, $C$ and the asymptotic tangent line at $g$ 
{\rm (Figure~\ref{fig:godron-class_section})}. 
The actual configuration depends on which of the intervals defined by the 
exceptional values $\cos \frac{5\pi}{6}$, $\cos \frac{4\pi}{6}$, $\cos \frac{3\pi}{6}=0$, 
$\cos \frac{2\pi}{6}$, $\cos \frac{\pi}{6}$, $8/9$,  
the invariant $\rho$ belongs to, respectively\: 
\[
\begin{array}{lcl}
\rho\in(\,\,-\infty\,\,,\,\cos \frac{5\pi}{6}) 
&  \iff &  c_{\sP}<c_{\sV}<c_{\sT_-}<\sigma<c_{\sT_+}<c_{\sF}\,;\\
\rho\in(\cos \frac{5\pi}{6},\cos\frac{4\pi}{6}) 
&  \iff &  c_{\sP}<c_{\sV}<c_{\sT_-}<\sigma<c_{\sF}<c_{\sT_+}\,;\\
\rho\in(\cos \frac{4\pi}{6},\cos\frac{3\pi}{6}) 
&  \iff &  c_{\sP}<c_{\sV}<c_{\sT_-}<c_\sF<\sigma<c_{\sT_+}\,;\\
\rho\in(\cos\frac{3\pi}{6},\cos\frac{2\pi}{6})  
& \iff & c_\sP<c_\sF<c_{\sT_-}<c_\sV<\sigma<c_{\sT_+}\,;\\
\rho\in(\cos \frac{2\pi}{6},\cos \,\frac{\pi}{6}\,)  
& \iff &  c_\sP<c_\sF<c_{\sT_-}<c_\sV<\sigma<c_{\sT_+}\,;\\
\rho\in(\cos \,\frac{\pi}{6}\,,\frac{8}{9})  
&  \iff &  c_\sP<c_{\sT_-}<c_\sF<c_\sV<\sigma<c_{\sT_+}\,;\\
\rho\in(\frac{8}{9},1) &  \iff &  c_{\sT_-}<c_\sP<c_\sF<c_\sV<\sigma<c_{\sT_+}\,.
\end{array}
\]
\end{theorem}
\begin{figure}[h]
\centerline{\includegraphics[width=5.2in]{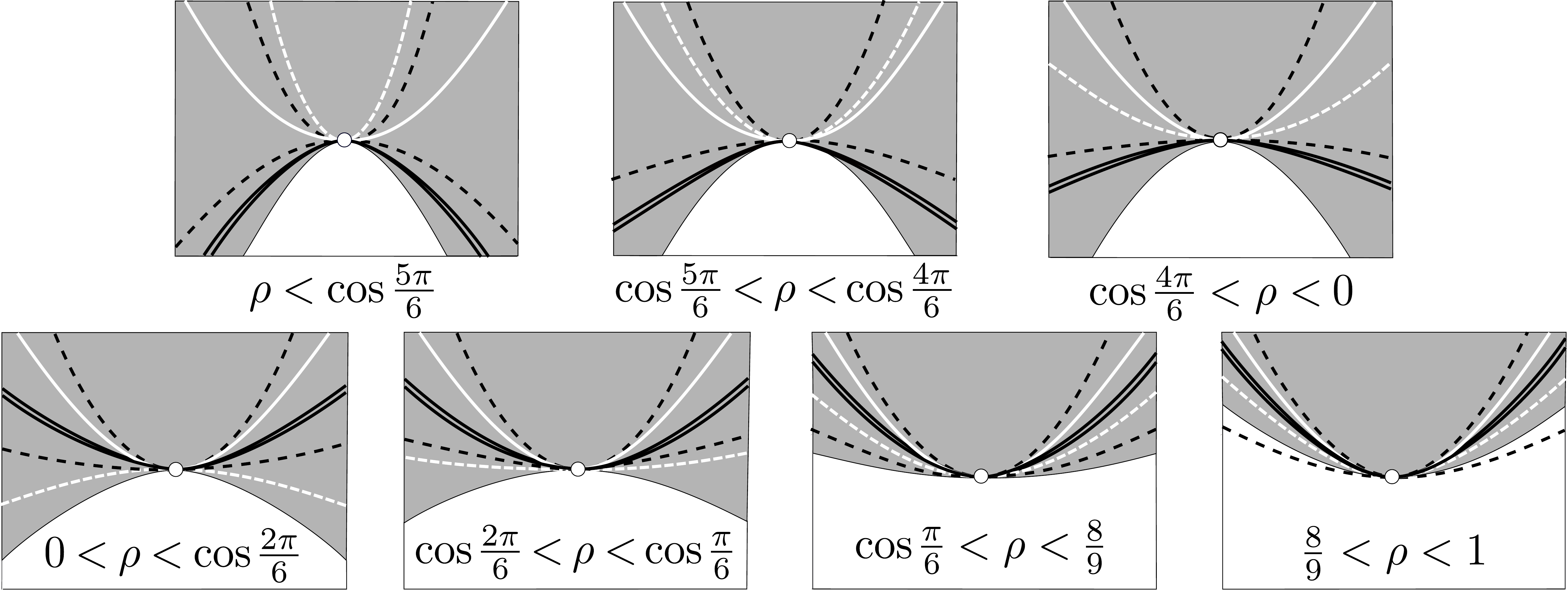}}
\caption{\small The seven generic configurations, at a hyperbolic cusp of Gauss, of the curves\,: 
flecnodal $F$ (white dotted), parabolic $P$ 
(boundary between white and grey domains), V-curve $V$ (black curves
which are very close together), tangent section 
$T_\pm$ (black dotted curves), line of principal curvature $C$ (white).}
\label{fig:godron-class_section}
\end{figure}

\begin{proof}
To determine the flecnodal curve $F$ near $g$ we consider 
tangent lines to $M$ at points $(x,y,f(x,y))$ and impose
the condition that the line should have at least 4-point contact with $M$. 
It is then straightforward to calculate the local equation:
\begin{equation}
\mbox{flecnodal curve} \ F: \ \ y =\ha \sigma\rho(2\rho-1)x^2+\ldots \, .
\label{eq:flec-hypCoG}
\end{equation}

The local equation of $P$, $y=\ha\sigma(3\rho-2)x^2+\ldots$\,, is given by the Hessian: 
$f_{xx}f_{yy}-f_{xy}^2=0$. We get the local equations of $T_\pm$ from \eqref{eq:hypCoG} 
by solving $f(x,y)=0$: $y=\ha\sigma(1\pm\sqrt{1-\rho})+\ldots$.  

If in addition we use \eqref{eq:flec-hypCoG}, \eqref{eq:vt-CoG} and Proposition~\ref{prop:curvature-separating},  
we find that the $2$-jets of the curves $F$, $P$, $V$, $T_-$, $T_+$ and $C$ on $M$, are curves in the 
tangent plane written as $y=h(x)$, where $h$ is given by the following respective functions\,: 
\[\ha\sigma\rho(2\rho-1)x^2, \quad \ha\sigma(3\rho-2)x^2, \quad \ha\sigma\rho x^2, 
\quad \ha\sigma(1+\sqrt{1-\rho})x^2, \quad \ha\sigma(1-\sqrt{1-\rho})x^2, \quad \ha\sigma x^2\,. 
\]
Thus the respective curvatures are $c_{\sF}=\sigma\rho(2\rho-1)$, $c_{\sP}=\sigma(3\rho-2)$, 
$c_{\sV}=\sigma\rho$, $c_{\sT_-}=\sigma(1+\sqrt{1-\rho})$, $c_{\sT_+}=\sigma(1-\sqrt{1-\rho})$ and 
$c_{\sC}=\sigma$. Since all these curvatures have $\sigma$ as factor, their relative positions 
in the real line are determined by $\rho$.  Thus in Figure~\ref{fig:godron-curvatures} the 
curvatures are divided by $\sigma$. 

\begin{figure}[h]
\centerline{\includegraphics[width=5in]{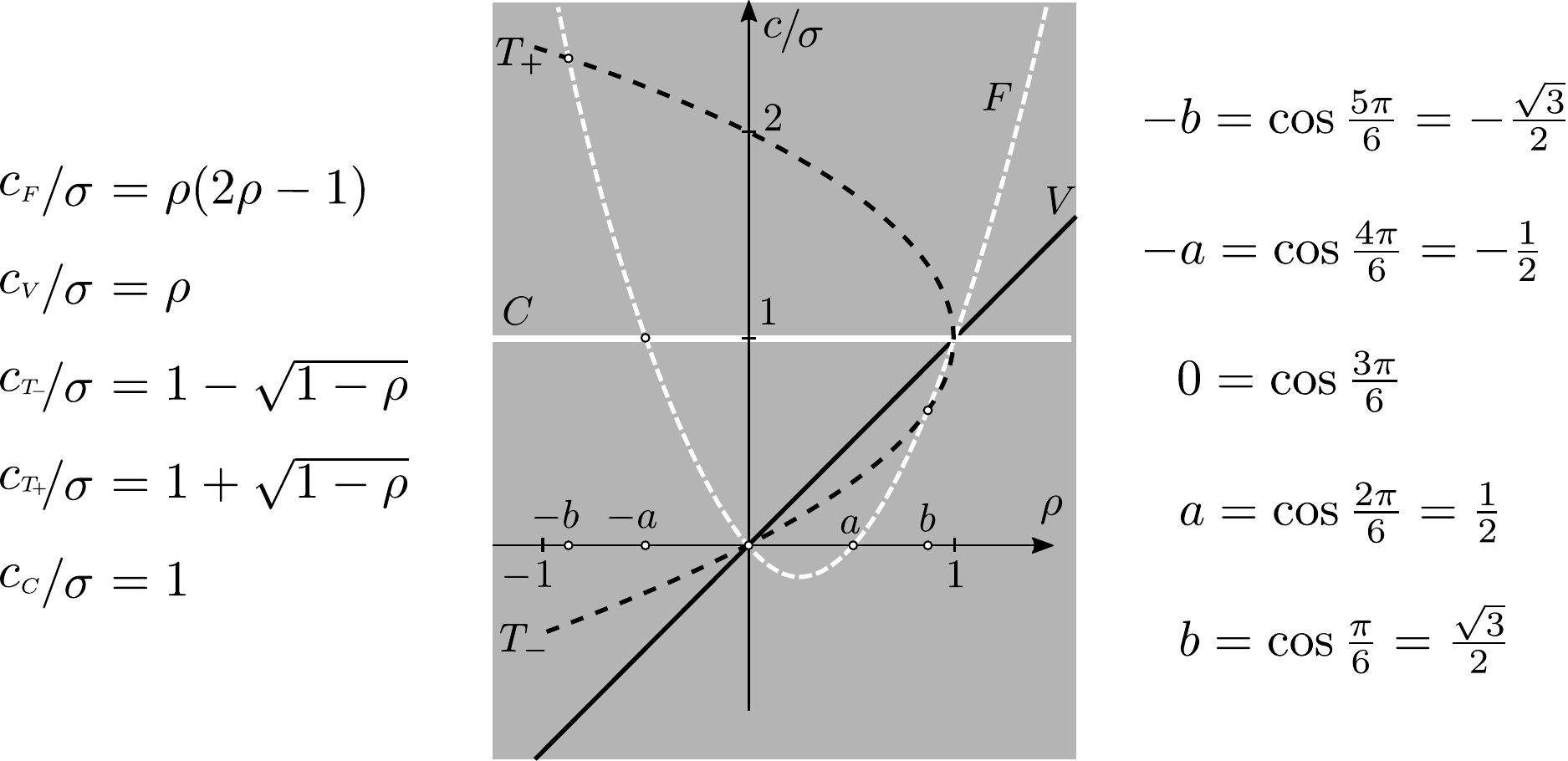}}
\caption{\small The curvatures $c_F$ (white dotted), $c_{\sV}$ (black), $c_{\sT_\pm}$ (dotted) 
and $c_{\sC}$ (white), all divided by $\sigma$.}
\label{fig:godron-curvatures}
\end{figure}
The expressions for the curvatures determine the exceptional values of $\rho$. 
\end{proof}
\noindent
\textbf{\small Note}. 
{\small For simplicity we omit from Figure~\ref{fig:godron-curvatures} the graph of 
$c_{\sP}/\sigma=(3\rho-2)$, which is a line. It cuts $c_{\sT_-}$ at $\rho=\frac{8}{9}$ 
and the $\rho$-axis at $\rho=\frac{2}{3}$ (the value of $\rho$ where $P$ changes its convexity).} 


\subsection{Relative positions of the flecnodal and vertex curves considering left, right branches and minimum, maximum types}
Consider a hyperbolic cusp of Gauss $g\in M$ with given cr-invariant $\rho$ and separating invariant $\sigma$, 
and write $r_1 = -\ha\sigma(1+\sqrt{1-\rho})$,\, $r_2=-\ha\sigma(1-\sqrt{1-\rho})$. 

Take $M$ in Monge form \eqref{eq:hypCoG}. 
\begin{prop}\label{prop:maxmin-CoG}
At points close to $g$ on the tangential components $V^1, V^2$ 
of the V-curve the absolute radius function has
\begin{eqnarray*} \mbox{a maximum on} \ V^1 &\Longleftrightarrow&  G_1>0\,; \\
 \mbox{a maximum on} \  V^2 &\Longleftrightarrow & \rho\,G_2>0 \quad (G_2, \ \rho \ \mbox{have equal signs})\,,
\end{eqnarray*}
where $G_1:=-r_1^2b_2+r_1c_1-d_0$\, and \,$G_2:=-r_2^2b_2+r_2c_1-d_0$. 
\end{prop}
Note that this is a different use of the notation $G_1, G_2$ from \S\ref{ss:contact2}.
\begin{proof}
We use Proposition~\ref{prop:maxmin-hyper} to find the conditions for the points of 
these two branches to represent maximum/minimum of the absolute radius function. 
The function $r$ on the two tangential components $V^1$, $V^2$ takes the respective forms
$\widehat{r}_1 = r_1+ \ldots$\, and\,  $\widehat{r}_2=r_2 + \ldots$\,. 

The function $\partial^3H/\partial x^3$, on $V^1$ and $V^2$, has the respective signs of 
$G_1$ and $G_2$. 

According to Proposition~\ref{prop:maxmin-hyper} a point on $V^i$ has a maximum if and only if 
$\widehat{r}_iG_i<0$, that is if and only if $r_iG_i<0$. Clearly $r_1<0$, so we get the last inequality 
if and only if $G_1>0$. On the other hand, 
it is easy to check that $r_2\rho <0$; hence 
$r_2G_2<0$ if and only if $\rho\, G_2>0$.  
\end{proof}
\begin{remark*}
{\rm
The quantities $G_1, G_2$ arise elsewhere. Consider the branches of the intersection $M\cap T_pM$ 
of $M$ with its tangent plane at $p$, that is the plane curve $f(x,y)=0$. The local equations are}
\[ y = \ha\sigma\left(1+\sqrt{1-\rho}\right)x^2+\textstyle{\frac{1}{\sigma\sqrt{1-\rho}}}\displaystyle G_1x^3 + \ldots, \ \ \ 
y = \ha\sigma\left(1-\sqrt{1-\rho}\right)x^2 - \textstyle{\frac{1}{\sigma\sqrt{1-\rho}}}\displaystyle G_2x^3 + \ldots. \]
\end{remark*}

\begin{prop}\label{prop:relative-position_V1-V2}
Close to $g$ the tangential component $V^1$ is `below' $V^2$ for $x<0$ 
{\rm (that is, has lower $y$ values: $y_1<y_2$)} 
if and only if $G_1<G_2$.
\end{prop}
\begin{proof}
The relative size of $B_1$ and $B_2$ determines the relative position of $V^1$ and $V^2$. 
But we have: $G_1-G_2=-\sigma\sqrt{1-\rho}\,(\sigma b_2+c_1)$. Thus by (\ref{eq:vt-CoG}), 
$G_1<G_2$ if and only if $B_2<B_1$. 
\end{proof}


\begin{theorem}
There are $12$ generic types of hyperbolic cusps of Gauss, according to the 
relative positions of the flecnodal curve branches $F_r$, $F_\ell$ and of 
the branches $V^1_r$, $V^1_\ell$, $V^2_r$, $V^2_\ell$ of the V-curve, 
counting their maximum and minimum types. These types are listed in Figure~{\rm \ref{fig:CoG1}}.  
\end{theorem}
\begin{proof}
The relative position `below/above' between the tangential components $V^1$, $V^2$ 
of the V-curve are given by the inequalities $G_1<G_2$ and $G_2<G_1$ 
(Proposition~\ref{prop:relative-position_V1-V2}); each one has three 
realisations, for example: $G_1<G_2<0$, $G_1<0<G_2$, $0<G_1<G_2$. 
The two relative positions `below/above' between the V-curve and the flecnodal curve 
are given by the inequalities $\rho<0$ and $\rho>0$ (Theorem\,\ref{rho-classification_tangent}). 
We get the type maximum or minimum for $V^1$ and $V^2$ from Proposition~\ref{prop:maxmin-CoG} applied 
to all these inequalities. Then we obtain the $12$ generic types shown in Figure~\ref{fig:CoG1}. 
\end{proof}

In Figure~\ref{fig:CoG1}, the flecnodal curve ($F_r$, $F_\ell$) is shown only 
in the left-hand diagrams. In all cases, the right and left branches of the two tangential components 
of the V-curve (denoted $1$, $2$) and of the flecnodal curve correspond to those indicated in the first diagram. 
Observe that for every position of $G_1$ and $G_2$, passing from $\rho<0$ to $\rho>0$ only the tangential 
component $V^2$ changes from maximum to minimum or vice versa. It is explained because if we pass from 
$\rho<0$ to $\rho>0$ continuously, at $\rho=0$ there is a flecgodron transition 
(Figure~\ref{fig:flecgodron-transition}) where only the component $V^2$ changes its type. 

\begin{figure}[H]
\centering
\includegraphics[scale=0.6]{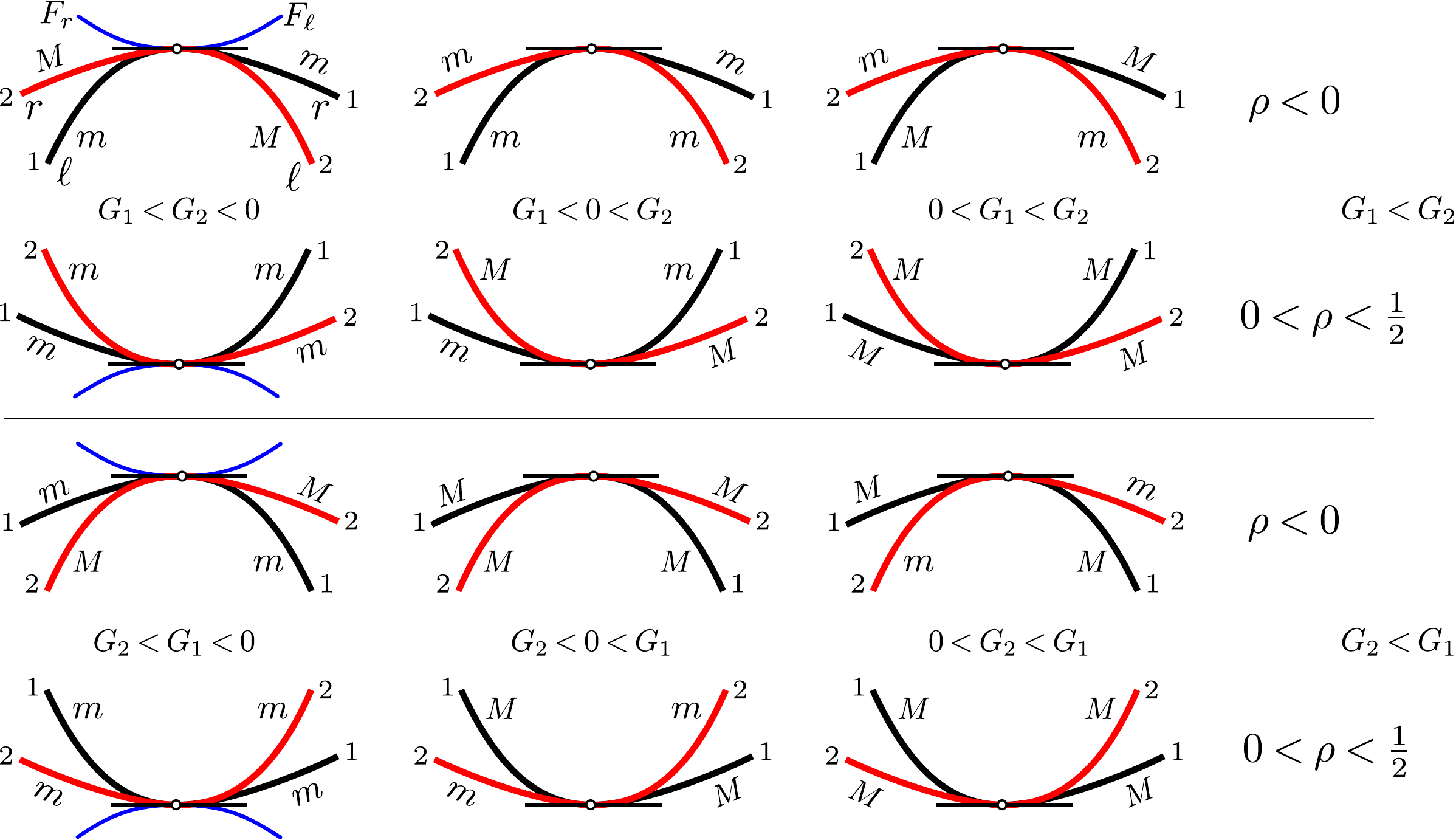}
\caption{\small The 12 generic local configurations of the flecnodal curve $F$ and of the 
components $V^1$, $V^2$ (denoted $1$ and $2$) counting their `left' ($\ell$), 
`right' ($r$) branches, and their maximum ($M$), minimum ($m$) types at a hyperbolic cusp of  Gauss. 
In all cases, the parabolic curve is `below' the other curves but is not drawn.}
\label{fig:CoG1}
\end{figure}

\section{Some codimension 1 transitions on the V-curve}\label{s:cod1}

In this section we shall investigate some transitions on the V-curve, and other curves, 
which occur in generic 1-parameter
families of surfaces, $\{M_t\}$, where $t$ is in some open interval of real numbers containing $t=0$. 
We consider mainly transitions in which the parabolic curve undergoes a transition. The results of this
section are obtained by exact calculation for $M_0$ and 
are largely experimental for nearby members of the family.

We consider the following cases:
\begin{enumerate}
\item[\ref{s:cod1}.1] The parabolic set of the family $M_t$  is  undergoing a 
`Morse transition' \\
6.1a: the parabolic set of $M_0$ has an isolated point;\\
6.1b: the parabolic set of $M_0$ has a 
 self-intersection consisting of two transverse smooth branches. 
 
 These two cases are referred to as `non-transversal $A_3$ transitions' in~\cite{BGT1,BGT2}, the
$A_3$ referring to contact between the tangent plane at the origin and $M_0$. Case 6.1a is $A_3^+$ and 6.1b is $A_3^-$.
\item[\ref{s:cod1}.2] The parabolic set is undergoing a `$D_4$ transition' as in ~\cite{BGT1,BGT2}. 
The symbol $D_4$ refers to the contact between the surface and its tangent plane at the origin and means
that the quadratic terms in the Monge form for $M$ vanish identically.  The cubic terms can have
one real root ($D_4^+$) or three ($D_4^-$).
\item[\ref{s:cod1}.3] $M_0$ has a degenerate cusp of Gauss, that is in the notation of Definition~\ref{def:coG}
and (\ref{eq:coG}), $b_0=0, \ c_0=\frac{1}{4}\sigma^2$.  This means that the contact of $M$ with its tangent
plane at the origin is of type at least $A_4$.  To ensure that the contact is no
higher than $A_4$ we require $\sigma^2b_2+2\sigma c_1+4d_0\ne 0$. 
Such a point of $M$ is called a `bigodron' in~\cite{ricardo} and an
`$A_4$ transition' in~\cite[\S3.2]{BGT1}; it occurs when an elliptic and a hyperbolic 
cusp of Gauss come into coincidence and disappear and is generic in a 1-parameter family of surfaces.
The parabolic curve remains nonsingular throughout.

\item[\ref{s:cod1}.4] The V-curve is {\em singular} because both components of (\ref{eq:singV}) are zero; this amounts
to saying that $d_0$ and $d_1$ are expressible in terms of coefficients $a, b_i, c_j$ in the hyperbolic case.

\item[\ref{s:cod1}.5] In \S\ref{ss:flecgodron} we describe a different kind of transition,
the `flecgodron' which is the coincidence of a biflecnode and a cusp of Gauss.

\end{enumerate}

\bigskip

The generic transitions of the parabolic set of a surface in 3-space are enumerated in~\cite{BGT1,BGT2}, the second
of these articles providing full mathematical details of results summarized in the first.   Since
the V-curve does not intersect the parabolic set except at hyperbolic  cusps of Gauss, the cases \ref{s:cod1}.1--2
are restricted to those where hyperbolic cusps of Gauss are created or destroyed.  For \ref{s:cod1}.1a this means
that an `elliptic island' appears in a hyperbolic region of $M_t$ as $t$ passes through 0, in which case two
hyperbolic cusps of Gauss are created on a newly created closed curve of the parabolic set.  See Figure~\ref{fig:singular-parabolic}.
For \ref{s:cod1}.2 only one of the two cusps of Gauss which are created or destroyed in the transition can intersect the V-curve.

\begin{figure}[H]
\centerline{\includegraphics[width=5in]{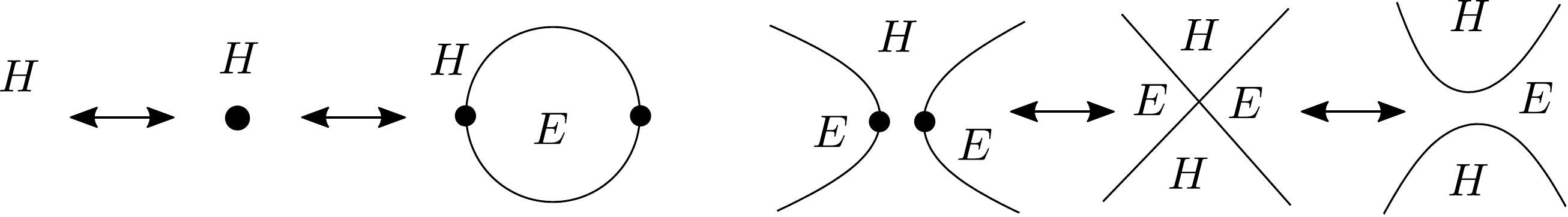}}
\caption{\small A diagram of local transitions on the parabolic curve $P$ in a 
generic {\color{black}$1$-parameter} family of surfaces as in \ref{s:cod1}.1a,b (from \cite[Figure 2]{BGT1}), 
where $H,E$ denote the hyperbolic and elliptic regions respectively.  \\
(left): from empty to a closed loop of $P$ which has two cusps of Gauss (the dots); \\
(right): through a crossing of smooth branches in which two cusps
of Gauss are created or destroyed. \\ 
These are the cases in which {\em hyperbolic} cusps
of Gauss are involved, so that the V-curve is involved too.
 }
\label{fig:singular-parabolic}
\end{figure}

\subsection{`Morse ($A_3$) transition' on the parabolic curve}\label{ss:Morse}

\begin{theorem}
Assume the parabolic set of a generic family of smooth surfaces 
$M_t$ has an $A_3$ (Morse) transition (at $t=0$) in a point $p$ of the surface $M_0$. Then 
\smallskip  

\noindent
$(a)$ If the parabolic set of $M_0$ locally consists of $p$, then the $V$-curve also consists of $p$.
\smallskip

\noindent
$(b)$ If the parabolic set of $M_0$ has two transverse smooth branches at $p$, then 
both components of the $V$-curve also consist of a crossing of two smooth branches at $p$.
\smallskip

\noindent
$(c)$ In case $b$, the two components of the $V$-curve have the same pair of tangent lines, 
and these tangents are distinct from the tangents to the parabolic curve. 
\end{theorem}

\begin{proof}
For a surface $M_0$ in Monge form (9) the parabolic curve has local equation 
\begin{equation}\label{P-local-eq}
3b_0x+b_1y+(3b_0b_2-b_1^2+6c_0)x^2+(9b_0b_3-b_1b_2+3c_1)xy+(3b_1b_3-b_2^2+c_2)y^2=0,  
\end{equation}
up to order $2$ in $x$, $y$. It is singular provided $b_0=b_1=0$ so that the Monge form becomes 
\[z=y^2+b_2xy^2+b_3y^3+c_0x^4+c_1x^3y+c_2x^2y^2+c_3xy^3+c_4y^4+d_0x^5+\ldots \,.\]

The parabolic curve has a Morse (i.e., nondegenerate) singularity provided the discriminant 
$\Delta:=8b_2^2c_0-8c_0c_2+3c_1^2$ of the quadratic terms in \eqref{P-local-eq} is nonzero. So we get    
\begin{equation}\label{cross-isolated}  
\mbox{a crossing for } \Delta>0, \qquad \mbox{an isolated point for } \Delta<0\,. 
\end{equation} 

To examine the V-curve of $M_0$, we use the reduced contact function $H$ (in Definition 2.1) 
and find that the values of $r$ for $5$-point contact at the origin are given by $r^2+c_0=0$. 
Thus we require $c_0<0$ for real 5-point contact circles ($c_0=0$ is of higher codimension). 
In this case, write $C_0=\sqrt{-c_0}$ so that $r=\pm C_0$ and $\Delta=-8b_2^2C_0^2+8C_0^2c_2+3c_1^2$. 

The local equations of both components of the $V$-curve (for the corresponding values $r=\pm C_0$) 
have the same $2$-jet\,:  
\begin{equation}\label{loc-eq-V-curve}
8C_0^4x_0^2-4C_0^2c_1x_0y_0+(4C_0^2b_0^2-4C_0^2c_2-c_1^2)y_0^2=0\,.
\end{equation}
Its discriminant, \,$\widehat{\Delta}:=16C_0^4(-8b_2^2C_0^2+8C_0^2c_2+3c_1^2)$,\, is a 
positive multiple of the discriminant $\Delta$ of the $2$-jet of the local equation of 
the parabolic curve, \,$\widehat{\Delta}=16\,C_0^4\Delta$. 

Thus using this last equality together with \eqref{cross-isolated} we prove items $a$ and $b$.  
\smallskip

\noindent
\textit{Proof of item $c$}: The pairs of tangents of the two components of the $V$-curve are defined by the 
same quadratic form \eqref{loc-eq-V-curve}, and one can check that the coincidence of 
these tangents with the tangents to the parabolic curve would lead to $\Delta=0$. 
\end{proof}

\noindent
\textbf{Remark}. Note that $c_0<0$ is also the condition for the intersection between $M_0$ and its tangent plane 
$T_pM_0$ at $p$ to consist of two tangential curves---a tacnode---rather than an isolated point. 

\smallskip

 The transitions occurring on the V-curves in a generic family of surfaces $M_t$ with $M_0$ as above are illustrated in
 Figure~\ref{fig:singular-parabolic-transition}.
 
 \smallskip\noindent
 \begin{rem}\label{note:blue-red}
 {\bf Note on the Figures~\ref{fig:singular-parabolic-transition}, \ref{fig:D4minus},
 \ref{fig:D4plus1branch}, \ref{fig:D4plus3branches}, \ref{fig:D4plus5branches} and
 \ref{fig:A4}.}
In these diagrams the radius of the 5-point contact circle
varies continuously along a segment of the V-curve with a single colour. 
Thus at a cusp of Gauss, the branches denoted by
$V^1$ and $V^2$ in (\ref{eq:vt-CoG}) and
Proposition~\ref{prop:left-righ_Vcurve-at-g} 
maintain the same colour on each side, and the left-rightness changes.
Of course, left-rightness stays constant along a segment of the V-curve 
with a single colour away from
cusps of Gauss.
\end{rem}
 
\begin{figure}[!h] 
\begin{center}
\includegraphics[width=5.5in]{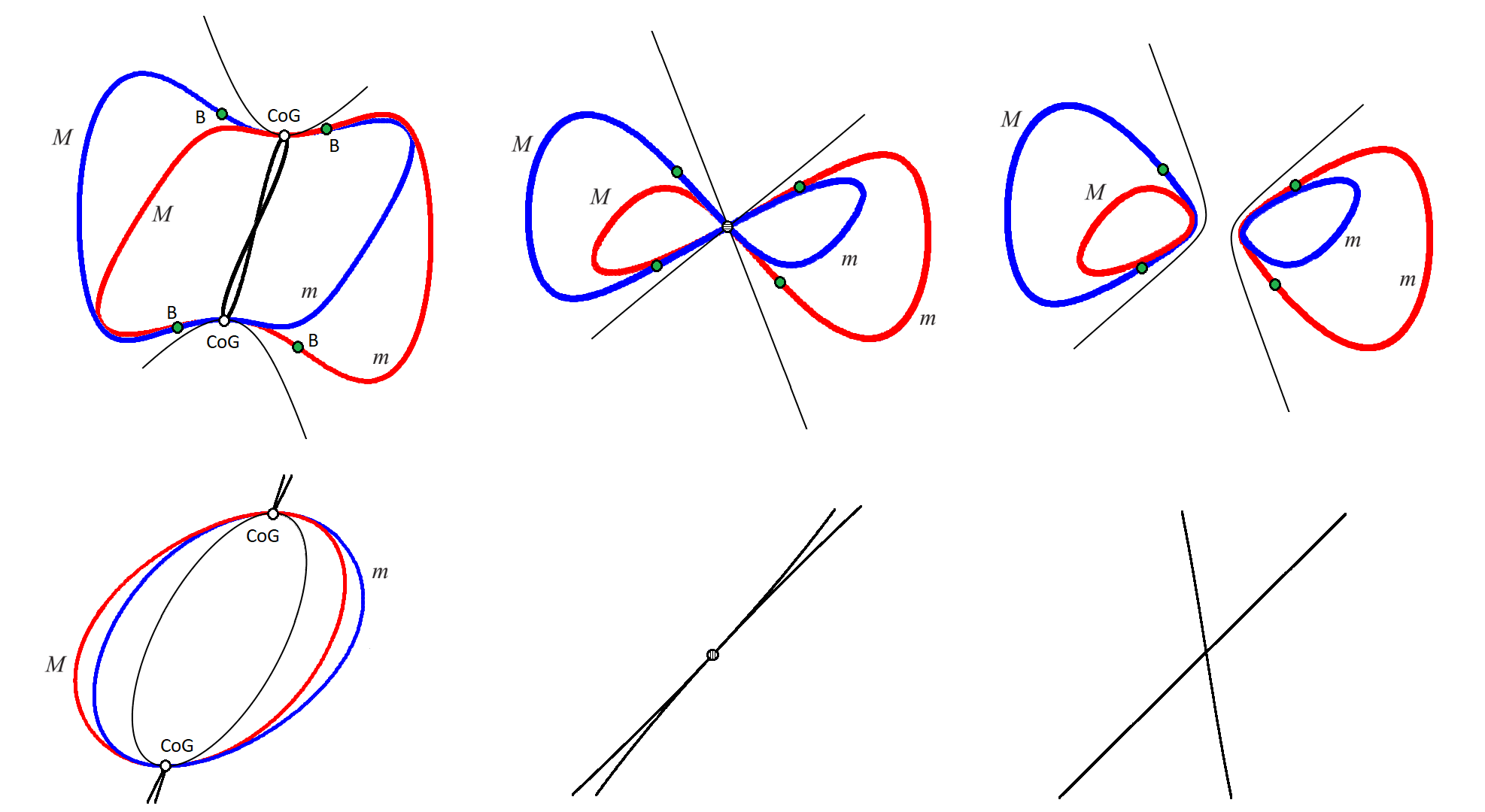}
\end{center}
\caption{\small Transitions on the parabolic set (thin black), the flecnodal curve (thick black) and V-curve (red and blue) in a generic
family of surfaces $M_t$ as the parabolic set undergoes a Morse transition. Here the $m$ or $M$ labelling of the V-curves indicates whether the relevant 5-point contact
circle is a minimum or maximum of the absolute radius of curvature as described in Section \ref{ss:maxmin}.  The small circles (labelled B in the top left diagram)
represent bi-vertices which separate the minimum and maximum components: bi-vertices are not involved in the transition. CoG stands for (hyperbolic) cusp of Gauss.
 Above: a crossing on the parabolic set,
case~\ref{s:cod1}.1b; below: an isolated point, case~\ref{s:cod1}.1a. The family of surfaces in the upper diagram is
$z = y^2 - x^4 + x^3y + x^2y^2  + tx^2  $ for small $t$ with $t=0$ in the middle, and 
$z = y^2 - x^4+ x^3y - x^2y^2  + tx^2$ for the lower diagram.}
\label{fig:singular-parabolic-transition}
\end{figure}

\subsection{$D_4$ transition on the parabolic curve (``flat umbilic'')}\label{ss:D4}

The label `$D_4$' refers to contact between $M_0$ and its tangent plane to $M_0$ at the origin.
In this case the quadratic terms of the Monge form of the surface $M_0$ are absent. Such a 
point is also called a {\em flat umbilic} in contrast to a generic umbilic which has quadratic
terms of the form $\kappa(x^2+y^2), \kappa\ne 0$.

By rotating the coordinates and scaling equally in all directions we may assume 
that the Monge form of
$M_0$ is
\begin{equation}
z=x^2y+b_2xy^2+b_3y^3+c_0x^4+c_1x^3y+c_2x^2y^2+c_3xy^3+c_4y^4+\ldots,
\label{eq:MongeD4}
\end{equation}
with one root of the cubic terms along the $x$-axis $y=0$. The two cases are distinguished by
\[ D_4^+ (\mbox{one real root}): b_2^2<4b_3; \qquad \ D_4^- (\mbox{three real roots}): b_2^2>4b_3.\]
We assume from now on that $b_2^2 \ne 4b_3$, so that the contact between $M_0$ and the
tangent plane $T_pM_0$ at {\color{black}$p$} is `no worse' than $D_4$.

The parabolic curve of $M_0$ has the form $x^2+b_2xy+(b_2^2-3b_3)y^2 + \mbox{h.o.t.} = 0$
with discriminant of the quadratic terms equal to $-3(b_2^2-4b_3)$. Thus (see~\cite[p.298]{BGT1},
noting that the labels $D_4^\pm$ on Figure~4 are the wrong way round):\\
$D_4^+$: one branch of $M_0\cap T_pM$  and the parabolic curve has a crossing of smooth branches\\
$D_4^-$: three transverse branches of $M_0\cap T_pM_0$ and the parabolic curve has an isolated point.

Looking for circles in the $(x,y)$-plane which have 5-point contact with $M_0$ at the origin we find that
for each real branch of $z=0$ there is a circle centred on the line perpendicular to that branch;
for the branch tangent to $y=0$ the centre of this circle is at $\left(0,-\frac{1}{2c_0}\right)$
(and the radius is of course
$\frac{1}{2|c_0|}$). 

{\em We shall make the generic assumption $c_0\ne 0$ in what follows.}

 But the second circle having 5-point contact and centre $(0,r)$ {\color{black}shrinks to} a 
 `circle of radius {\color{black}$r=0$}'.  We shall see in \S\ref{sss:radius0}
that indeed
there is at least one branch of the V-curve through the origin on which 
the radius of the 5-point contact circle tends to 0 at the origin.

To analyse this situation for the transitional surface $M_0$ we shall adopt the `alternative' approach to
the contact function as described in~\S\ref{ss:contact2}, using the mapping
$H=(H_1,H_2)$ described there.  
The V-curve consists of those points $(x_0, y_0, f(x_0,y_0))\in M_0$ for which $u,v$ exist such that $H$
is $\mathcal K$-equivalent to an $A_{\ge 4}$ singularity at $p=q=0$. 

\subsubsection{{\color{black}Osculating} circle of radius $1/|2c_0|$}

We consider here the branch of the intersection of $M_0$ with its tangent plane $z=0$ at the origin
which is tangent to $y=0$.  For $D_4^+$ this is the only real branch of the intersection while for
$D_4^-$ the same argument applies to each of the three real branches of the intersection.

Consider the circle having 5-point contact with $M_0$ at the origin and centre  
$(u,v)=\left(0,-\frac{1}{2c_0}\right)$. We shall expand $H$ about $(x_0,y_0,u,v,p,q)$ 
$=
\left(0,0,0,-\frac{1}{2c_0},0,0\right)$, substituting $v=V-\frac{1}{2c_0}$ so that $V$ is small.
The coefficient of $q$ in $H_2$ then works out as $\frac{1}{c_0}\ne 0$ so that we can
solve $H_2=Q$ say for a function $q=Q(p,x_0,y_0,u,V)$.  Substituting in $H_1$ we can
then put $Q=0$ in $H_1$ since we are classifying $H$ up to contact equivalence and terms
containing $Q$ can therefore be removed from $H_1$. The result is a mapping
$(\overline{H}_1, Q)$ say where $\overline{H}_1$ is a function of $p,x_0,y_0,u,V$.
It is a straightforward matter to check that $\overline{H}_1$ is divisible by $p^2$,
corresponding to the fact that the circle always has at least 2-point contact with $M_0$
at $(x_0,y_0,f(x_0,y_0))$. Then the second, third and fourth derivatives of $\overline{H}_1$
evaluated at $p=0$ give 3 equations in $x_0,y_0,u,V$ the solution to which is the preimage of one branch of the 
V-curve; the V-curve itself is the projection of this set to the $(x_0,y_0)$ plane.  In fact in this case
an argument similar to that in \S\ref{ss:hyper} shows that, provided $c_0\ne0$ as above, the 
solution set in $(x_0,y_0,u,v)$-space is smooth, parametrized locally by $x_0$, and its projection to the 
$(x_0,y_0)$ plane is therefore smooth.
In fact the V-curve is tangent to the $x$-axis, with local parametrization in the $(x_0,y_0)$ plane
of the form $y_0=-2c_0x_0^2 + \mbox{h.o.t}$. 

Assume, in the notation of (\ref{eq:MongeD4}), that $b_2^2 \ne 4b_3$ and $c_0 \ne 0$. Then we have
the following.
\begin{prop}
In a generic $1$-parameter family of smooth surfaces $M_t$ having a $D_4^\pm$ transition at $t=0$, 
each smooth branch of $M_0\cap T_pM_0$ {\rm (the local intersection of $M_0$ with its tangent plane at $p$)} 
has a tangential smooth branch of the $V$-curve. 

The corresponding 5-point contact circle or circles at $p$ have nonzero radius; 
for the branch of $M_0\cap T_pM_0$ tangent
to the $x$-axis this radius is $1/|2c_0|$ in the notation of (\ref{eq:MongeD4}).
\label{prop:nonzeroradius}
\end{prop}
   
\subsubsection{Degenerate osculating circle of (limiting) radius 0}\label{sss:radius0}

The 5-point contact {\color{black}degenerate} circle {\color{black}of radius zero} is obtained by
expanding $H$ as power series in $x_0,y_0,u,v,p,q$, all of which are small. We find the following:
\begin{eqnarray}
H_1&=&
-y_0p^2-2(x_0+b_2y_0)pq-(b_2x_0+3b_3y_0)q^2-p^2q-b_2pq^2-b_3q^3+\mbox{
degree} \ge 4  \nonumber\\
H_2&=&-2(u-x_0)p-2(v-y_0)q+p^2+q^2 + \mbox{ degree} \ge 6. \label{eq:HD4}
\end{eqnarray}
The V-curve consists of those points $(x_0, y_0, f(x_0,y_0))\in M_0$ 
for which $u,v$ exists such that $H$
is $\mathcal K$-equivalent to an $A_{\ge 4}$ singularity at $p=q=0$.   It is convenient to
substitute $U=u-x_0, V=v-y_0$ so that the quadratic terms of $H_2$ take the form
$p^2+q^2-2pU-2qV$.

\begin{remark*}{\rm
Evaluating $H$ at $x_0=y_0=u=v=0$
we obtain $(-q(p^2+b_2q+b_3q^2)+\ldots, p^2+q^2+\ldots)$, which in the {\em complex}
$\mathcal K$ classification of~\cite{dimca-gibson}  is equivalent to the $\mathcal K$-simple
germ $B_{3,3}: (x,y)\mapsto (xy,x^3+y^3)$. According to the list of specializations
in~\cite[p.278]{dimca-gibson} this singularity has $A_5$ singularities in its neighbourhood; however
the unfolding by parameters $x_0, y_0, u, v$ will not be versal; in our situation of a
generic 1-parameter family of surfaces we do not
expect to  find more degenerate singularities than $A_4$.
}
\end{remark*}
We shall approach this case by neglecting terms of degree $\ge 6$ in $H_2$ and solving $H_2=Q$ say exactly
for $q$ as a function of $x_0, y_0, U, V, p, Q$ in order to reduce $H_2$ to $Q$.
When this is done we can replace $Q$ by 0 and $H$ takes the form $(\overline{H}_1, Q)$, say
where $\overline{H}_1$ is a function of $x_0, y_0, U, V, p$. Both this function and its derivative
with respect to $p$ vanish at $p=0$, since the circle always has at least 2-point contact with $M_0$
at $(x_0,y_0,f(x_0,y_0))$.  The conditions we want to impose are, as usual, that the second, third
and fourth derivatives vanish at $p=0$. 

When this is done we find (after a rather tedious calculation) that $x_0, y_0$ can be expressed in
terms of $U,V$ but that the relation between $U,V$ has lowest terms a homogeneous quintic:
\begin{eqnarray}
b_2(2b_2^2-7b_3)U^5+(4b_2^2b_3-5b_2^2-12b_3^2+14b_3)U^4V&&\nonumber \\
-b_2(2b_2^2-b_3-3)U^3V^2+(3b_2^2+6b_3-2)U^2V^3-5b_2UV^4+2V^5&=&0. \label{eq:quintic}
\end{eqnarray}
 A solution $(U,V)=(k,1)$ of this quintic gives a branch of the V-curve with slope 
\[ \frac{k(2-b_2k)}{3b_3k^2-2b_2k+1}\,,\]
in the $(x_0,y_0)$-plane. 
This slope cannot be zero provided $b_2^2\ne 4b_3$ as above, so that the V-curve branch 
cannot be tangent to the $x_0$ axis and by symmetry cannot be tangent to any of the branches 
of $M_0\cap T_pM_0$.

It can be shown that if $b_2^2 > 4b_3$ the quintic equation has negative discriminant (that is
for $D_4^-$), which indicates three real branches in the $(U,V)$-plane and therefore three
real branches of the V-curve in the $(x_0,y_0)$-plane, with slopes given as above by the three
real roots of the quintic. But if $b_2^2<4b_3 \  (D_4^+)$ there
can be one, three or five real branches of the V-curve. 

We sum up this situation as follows.
\begin{prop}
In addition to the smooth branches listed in Proposition~\ref{prop:nonzeroradius} the V-curve has 
other smooth branches, whose corresponding 5-point contact circle has radius $0$:\\
$D_4^-$: one for each of the three real branches of $M_0\cap T_pM_0$, in each case not tangent 
to the latter branch;\\
$D_4^+$: either 1, 3 or 5 such branches depending on the cubic terms of $M_0$ at the origin. 
For more information see below.
\label{prop:zeroradius}
\end{prop}

This situation can be better illustrated  by a change of normal form; in fact we shall adopt a procedure
analogous to that used to separate the ``lemon/star/monstar'' cases of a generic umbilic point, as 
explained in, for example,~\cite{Porteous}.  It is an elementary calculation to check that every real
cubic form in $x,y$ can be transformed, by rotation and scaling in the $(x,y)$-plane, and then
writing $z=x+\ii y$  ($\ii=\sqrt{-1}$), into the special form
\begin{equation}
 z^3 + 3\overline{\beta}z^2\overline{z}+3\beta z\overline{z}^2+ \overline{z}^3
 \label{eq:betacubic}
 \end{equation}
where $\beta$ is a complex number. The only exception is a cubic form $ax^3+bx^2y+axy^2+by^3$, which
equals $(ax+by)(x^2+y^2)$ and so has only one real root $ax+by=0$. So far as the case $D_4^+$ is
concerned we can ignore this exception, since the cubic form of $M_0$ has three real roots.  The 
conditions for (\ref{eq:quintic}) to have 1, 3 or 5 real roots can then be expressed in terms of $\beta=
\beta_1 + \ii\beta_2$
and the resulting diagram Figure~\ref{fig:umbilic-plot3} in the $\beta$ plane  
has a pleasing symmetry and compactness.

For the record, the result of expressing the quintic form (\ref{eq:quintic}) in terms of $\beta_1$ and
$\beta_2$ is
\begin{eqnarray}
3(\beta_1^3+\beta_1\beta_2^2-3\beta_1^2+5\beta_2^2+3\beta_1-1)U^5+
3\beta_2(\beta_1^2+\beta_2^2-1)U^4V&&\nonumber \\
+6(\beta_1-1)(\beta_1^2+\beta_2^2-1)U^3V^2+2\beta_2
(3\beta_1^2+3\beta_2^2+40\beta_1+17)U^2V^3&&\nonumber\\
+(3\beta_1^3+3\beta_1\beta_2^2-29\beta_1^2+11\beta_2^2+17\beta_1+9)UV^4+
\beta_2(3\beta_1^2+3\beta_2^2+16\beta_1+5)V^5.
\label{eq:betaquintic}
\end{eqnarray}
The discriminant of this quintic form is a product of two factors, one of which is the cube of
the discriminant of the cubic form (\ref{eq:betacubic})
of $M_0$ and the other has degree 10 in $\beta_1$ and $\beta_2$.
Both are invariant under the rotation $\beta\mapsto\beta\exp(2\pi\ii/3)$,
as is clear from Figure~\ref{fig:umbilic-plot3} below. 
Figures~{\color{black}\ref{fig:D4minus},} \ref{fig:D4plus1branch},
\ref{fig:D4plus3branches} illustrate the V-curve itself. The central
diagram in each case follows from the calculations above and the outer diagrams, representing
the evolution of the V-curve 
in a generic family of surfaces, are produced from an example.

\begin{figure}[H]
\begin{center}
\includegraphics[width=6.5in]{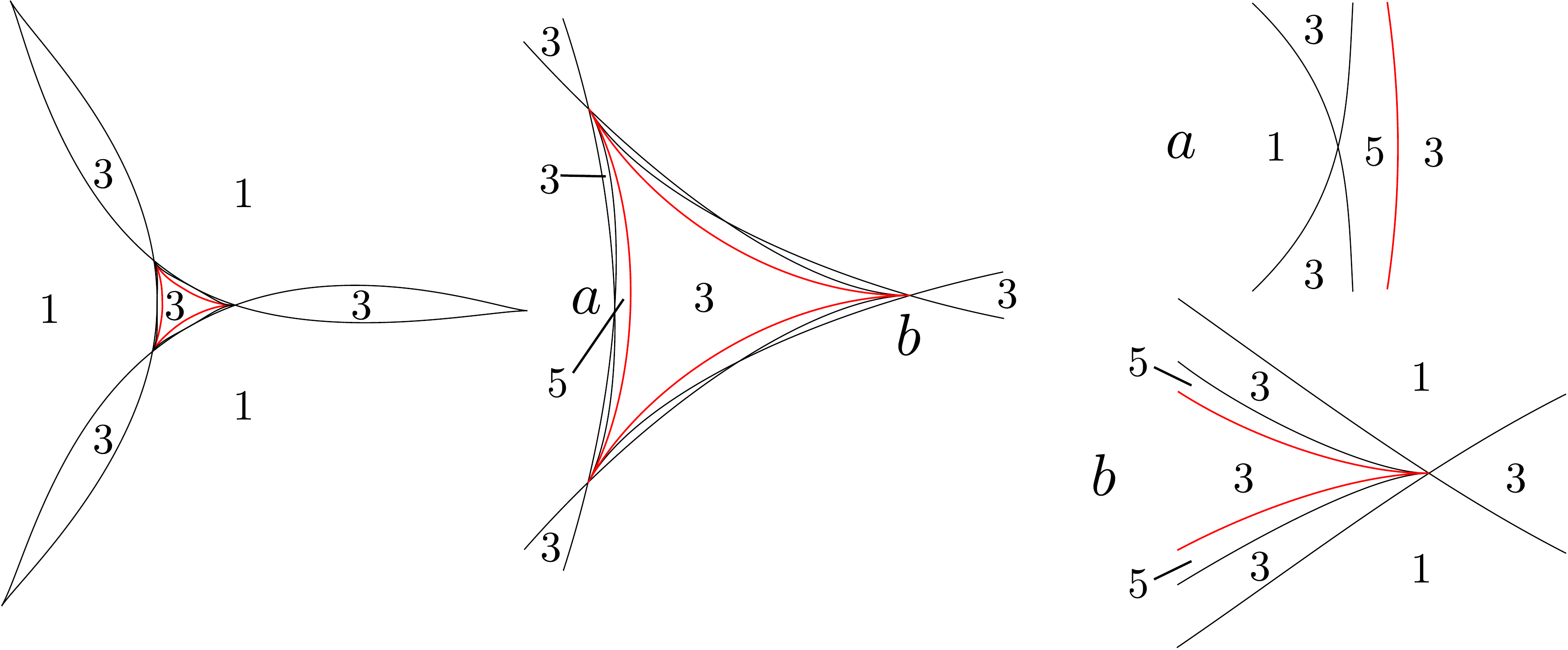}
\end{center}
\vspace*{-0.8cm}
\caption{\small The regions of the $\beta$-plane corresponding to different numbers 1,3,5 of real
roots of the quintic form (\ref{eq:betaquintic}).  The 3-cusped red curve is the discriminant of the cubic form
(\ref{eq:betacubic}) which forms part of the discriminant of (\ref{eq:betaquintic}). The middle
diagram is an enlargement of the central area of the left-hand diagram and the details of two 
small areas a and b of
the middle diagram are on the right. The region inside the red curve corresponds to $D_4^-$ and
all the rest of the diagram to $D_4^+$.}
\label{fig:umbilic-plot3}
\end{figure}

\noindent
{\bf Notation.}
In the following figures  the thin black curve is the parabolic curve, and the arrows show the direction of
increasing absolute {\color{black}radius} of the 5-point contact circle along the V-curve and  `Min' refers to a minimum
of this radius. In some figures, there 
are also 6-point contact points (bi-vertices) labelled $B$. The various thick lines are the branches of the V-curve;
see Remark~\ref{note:blue-red} for further details.
\begin{figure}[H]
\begin{center}
\includegraphics[width=6in]{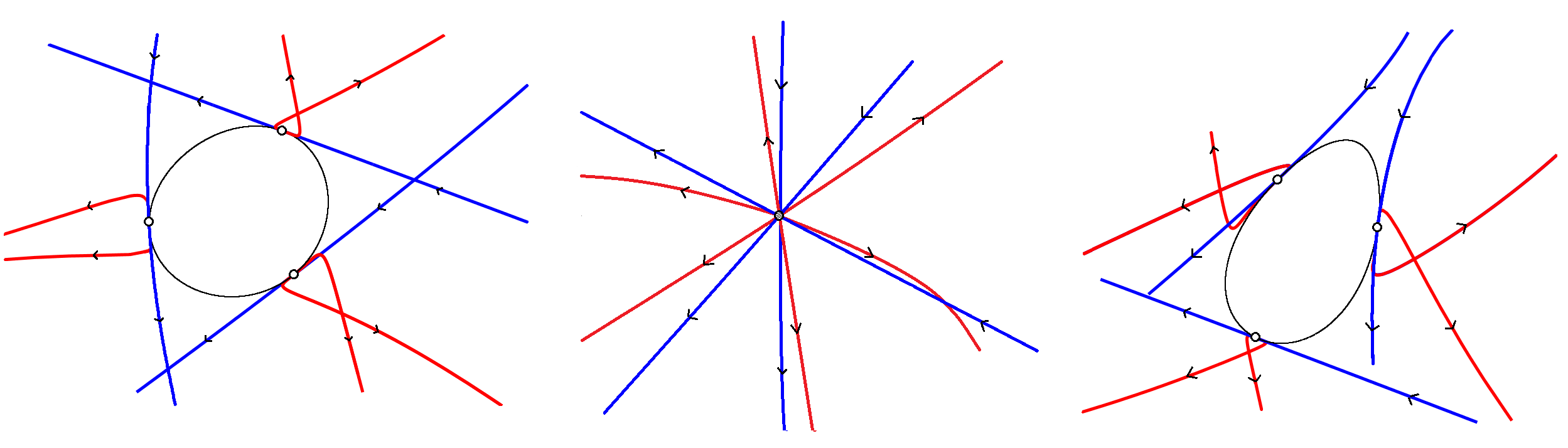}
\end{center}
\vspace*{-0.8cm}
\caption{\small $D_4^-$ case.  The surface $M_0$ in this example is 
$z= x^3-xy^2-\frac{4}{5}x^2y+\frac{1}{7}y^4$
 and the family $M_t$ is obtained by adding 
a small term $tx^2$.  For $M_0$ the blue curve is a branch of the
V-curve along which the radius of the 5-point contact circle is nonzero.  On the red curves near each cusp of Gauss there is a local minimum of radius and also a bi-vertex which are not shown.  The nearby flecnodal curves are also not included. }
\label{fig:D4minus} 
\end{figure}

\begin{figure}[H]
\begin{center}
\includegraphics[width=6in]{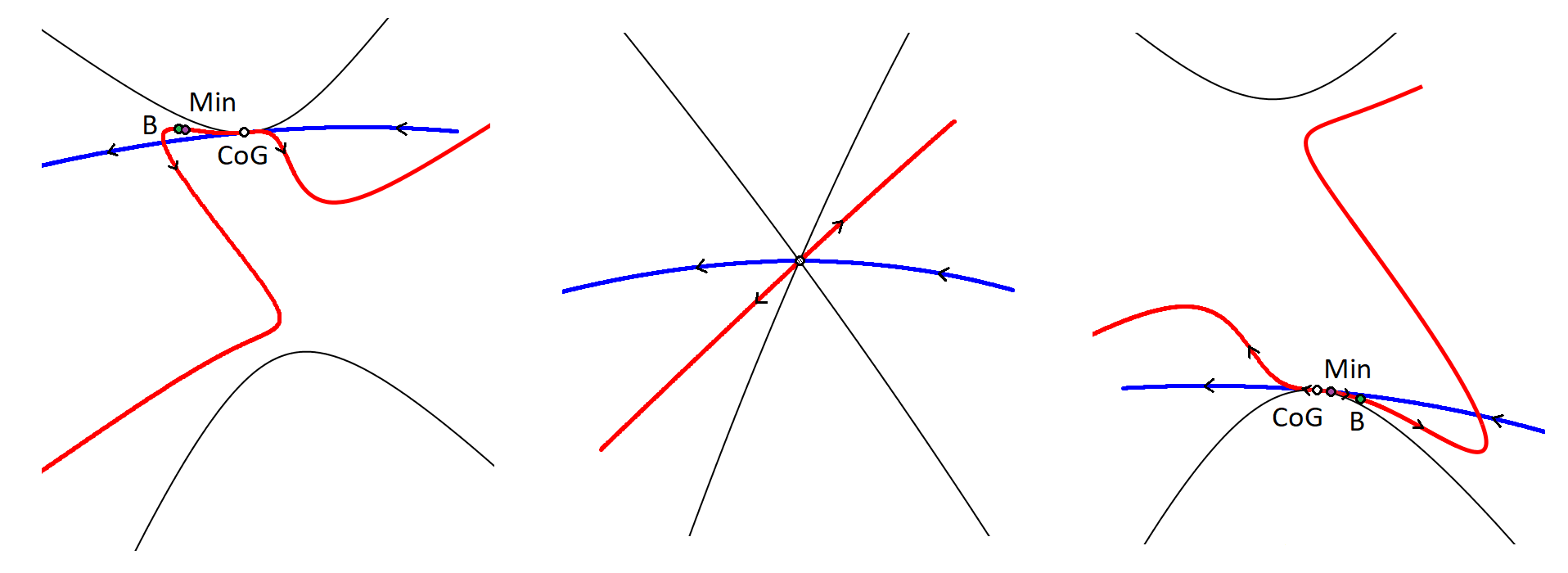}
\end{center}
\vspace*{-0.8cm}
\caption{\small $D_4^+$ case with one real root of the quintic form (\ref{eq:betaquintic}), 
in a family of surfaces $M_t$ where $t=0$ is the middle diagram.  For $M_0$
the blue curve is a branch of the
V-curve along which the radius of the 5-point contact circle is nonzero; the  other branch has this
radius with limit 0 at the crossing.   In this figure, the example surface $M_0$ chosen is  $z=x^3+xy^2 - \frac{4}{5}x^2y+\frac{1}{7}y^4$, 
and the family $M_t$ is obtained by adding a small term $t x^2$.}
\label{fig:D4plus1branch}
\end{figure}

\begin{figure}[H]
\begin{center}
\includegraphics[width=6in]{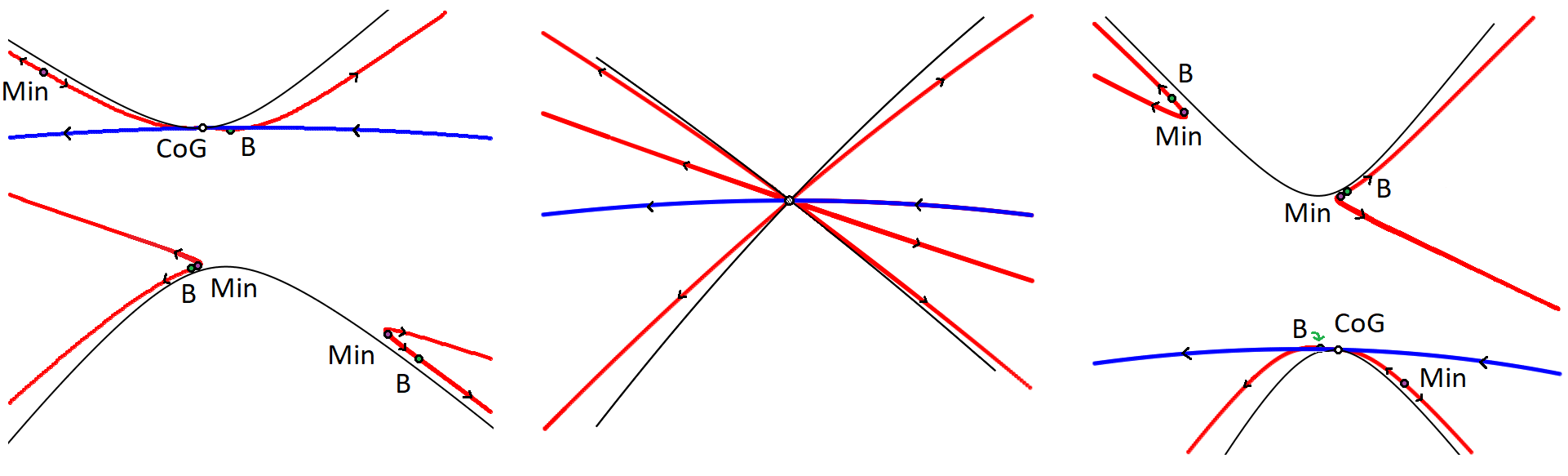}
\end{center}
\vspace*{-0.8cm}
\caption{\small $D_4^+$ case with three real roots of the quintic form (\ref{eq:betaquintic}), 
in a family of surfaces $M_t$ where $t=0$ is the middle diagram.  For $M_0$
the blue curve is the branch of the
V-curve along which the radius of the 5-point contact circle is nonzero; the three other branches all have this
radius with limit 0 at the crossing.  The surface   $M_t$  in this example is $z=x^2y + xy^2 + 4y^3 + 
\frac{1}{7}x^4 + tx^2.$ }
\label{fig:D4plus3branches}
\end{figure}

\begin{figure}[H]
\begin{center}
\includegraphics[width=6in]{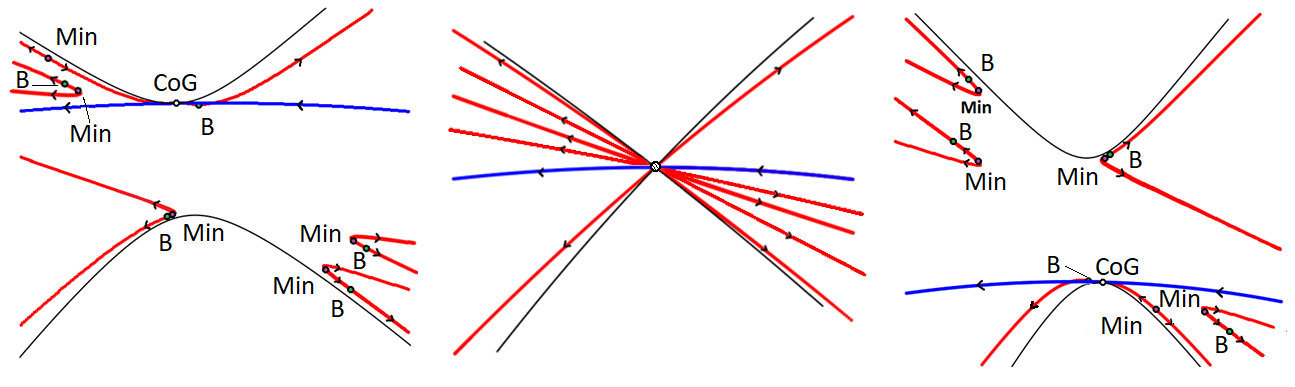}
\end{center}
\vspace*{-0.8cm}
\caption{\small $D_4^+$ case with five real roots of the quintic form (\ref{eq:betaquintic}), 
in a family of surfaces $M_t$ where $t=0$ is the middle diagram. The blue curve for $M_0$ is again
the branch of the
V-curve along which the radius of the 5-point contact circle is nonzero. The family of surfaces used
in this example is
$z = -x^2y + \frac{9}{2}xy^2 - \frac{51}{10}y^3 + \frac{1}{7}x^4 + ty^2$. }
\label{fig:D4plus5branches}
\end{figure}

\subsection{Double cusp of Gauss (bigodron)}\label{ss:degen-coG}

This case refers to higher ($A_4$) contact between $M_0$ and its tangent plane at the origin and the
Monge form of $M_0$ takes the form 
\begin{equation}
z=y^2-\sigma x^2y+b_2xy^2+b_3y^3+\qu\sigma^2x^4 + \ldots,
\label{eq:MongeA4}
\end{equation}
where $\sigma^2b_2+2\sigma c_1+4d_0\ne 0$.  In this case circles in the plane $z=0$
having 5-point contact with $M_0$ at the origin must have their centres on the $y$-axis
but there is only {\em one} solution to the position of the centre, namely $(0,\frac{1}{\sigma},0)$.
(Recall that $\sigma>0$.)
We shall see that this coincidence of solutions gives a smooth branch of the V-curve and also
a singular branch.

The parabolic curve has the form $y=\ha\sigma x^2+\ldots$; locally, the hyperbolic region is parametrized by $\{(x,y):y>0\}$.

Writing down the equations for the V-curve as usual we find in this case that $x_0$ and $y_0$ can
be expressed in terms of $u$ and $V=v-\frac{1}{\sigma}$ and that in the $(u,V)$-plane there is a locus
whose lowest terms take the form
\[0= \sigma^3uV+\frac{\sigma(6\sigma^3b_2^2+17\sigma^2b_2c_1+40\sigma b_2d_0+6\sigma c_1^2+20c_1d_0)}
{5(\sigma^2b_2+2\sigma c_1+4d_0)}u^2.\]
The denominator of the fraction is nonzero by the assumption of exactly $A_4$ contact of $M_0$ with
the plane $z=0$.  Thus there are two branches to the locus in the $(u,V)$-plane.  One of them
has the form $u=\mbox{constant}\times V + \ldots$ and the other $u=\mbox{constant}\times V^3+\ldots$, with in fact no term in $V^2$.
The first of these leads to a `parabola' component of the V-curve, of the form 
$y_0=\mbox{constant}\times x_0^2 + \ldots$ but the other gives a V-curve whose initial terms,
parametrized by $V$, are
\begin{eqnarray*}
x_0&=&-\frac{\sigma^4}{5(\sigma^2b_2+2\sigma c_1+4d_0)}V^2, \\ 
y_0&=&\frac{\sigma^9}{50(\sigma^2b_2+2\sigma c_1+4d_0)^2}V^4
-\frac{2\sigma^{10}(\sigma^2b_2+6\sigma c_1+20d_0)}{125(\sigma^2b_2+2\sigma c_1+4d_0)^3}V^5
\end{eqnarray*}
Of course, both components lie in the hyperbolic region of $M_0$.
\begin{remark*}
{\rm
Using $\mathcal{A}$-equivalence, the above singularity is not in fact equivalent to the
standard `rhamphoid cusp'
$(t^2, t^5)$ but to $(t^2, t^7)$.  Of course this equivalence is not Euclidean invariant. 
}
\end{remark*}
\begin{figure}[H]
\begin{center}
\includegraphics[width=6in]{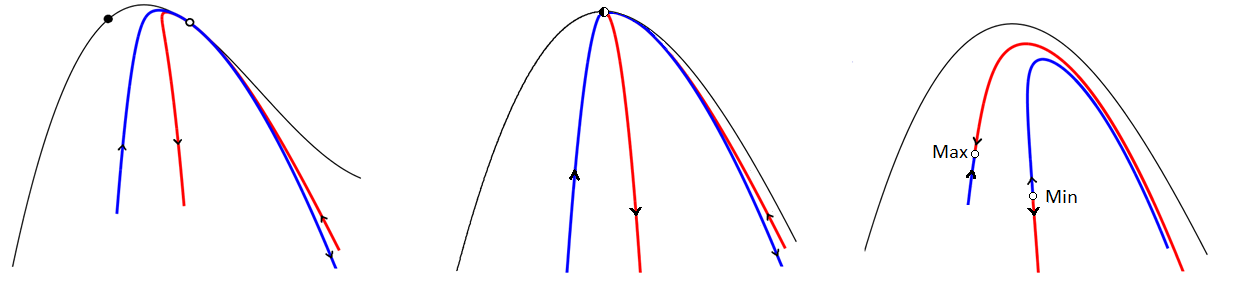}
\end{center}
\vspace*{-0.8cm}
\caption{\small $A_4$ transition on the V-curve.  The surface $M_0$ is
$z=y^2+x^2y+\frac{1}{4}x^4-\frac{1}{5}x^5-x^2y^2+2x^3y+3xy^3$, and the family
is given by adding a small multiple of $xy$. The thin black line is the parabolic curve and the thick coloured lines are the two branches of the V-curve; in the left-hand
diagram the colouring of the branches is in accordance with
 Remark~\ref{note:blue-red} and then
these branches are followed through the transition in the other two diagrams.  
On $M_0$ one branch has a `rhamphoid cusp' (the two close together curves, one red and one blue, on the right-hand side of the middle diagram) and the other branch is smooth. As before the arrows indicate the direction of increasing absolute radius of the 5-point contact circle. On the left figure the filled in circle is an elliptic cusp of Gauss and the white circle is a hyperbolic cusp of Gauss.  The flecnodal curve is not shown.}
\label{fig:A4}
\end{figure}

\subsection{Singular V-curve}
This is the special case where both components of the vector (\ref{eq:singV}) are zero.  
In that situation generically the V-curve will have a nondegenerate quadratic
form for its 2-jet, corresponding to the component of the intersection $M_0\cap T_pM_0$ 
tangent to the $x$-axis in the hyperbolic case.  Thus the V-curve
will have an unstable crossing or isolated point. 

In the case of a crossing this can be interpreted as saying the following.  
Corresponding to {\em one} of the branches of the intersection $M_0\cap T_pM_0$, 
having a $5$-point contact circle tangent to this branch, there are {\em two} distinct 
directions in which $p\in M_0$ can move away from the origin and still have a 
$5$-point contact circle tangent to the intersection of $M_0$ with its tangent plane at $p$. 
 
 In the case of an isolated point, which in the family $M_t$ will open out into a closed loop,
 this is a way in which the V-curve can acquire a single left or right loop, 
 in contrast to the situation depicted in Figure~\ref{fig:singular-parabolic-transition} 
 where {\color{black}two loops} appear on the V-curve.

This is not the same situation as {\color{black}a $V$-crossing}. There, {\em each} branch
of $M\cap T_pM$ contributes a smooth branch {\color{black}(one left and one right)} of the V-curve,  
whose crossing is stable under small perturbations of $M$. 

\subsection{Flecgodrons: the coincidence of a cusp of Gauss and a biflecnode}\label{ss:flecgodron}
At a simple cusp of Gauss $g$ with $\rho=0$ ($c_0=0$), exactly one value of $r$ (in Lemma\,\ref{prop:rs}) is zero. 
This means that the corresponding `circle' is a straight line having 
$5$-point contact with $M$ at $g$ (exactly $5$-point contact 
requires $d_0\neq 0$). Thus $g$ is a cusp of Gauss and is also a biflecnode. 
\medskip

\noindent
\textbf{\small Flecgodron}. A simple cusp of Gauss at which the asymptotic tangent line 
and the surface $M$ have $5$-point contact is called {\em a flec-godron}. 
\medskip

A surface in general position has no flecgodron\,: under any small generic deformation of 
a surface $M$ having a flecgodron the condition $\rho=0$ is destroyed. 
Perturbing $M$ inside a generic $1$-parameter family of surfaces (Figure~\ref{fig:flecgodron-transition}) 
$\{M_t\}$, there is an isolated 
parameter value $t_0$ (near $0$) whose corresponding surface $M_{t_0}$ has a simple flecgodron. 

\begin{figure}[h]
\centerline{\includegraphics[width=5.4in]{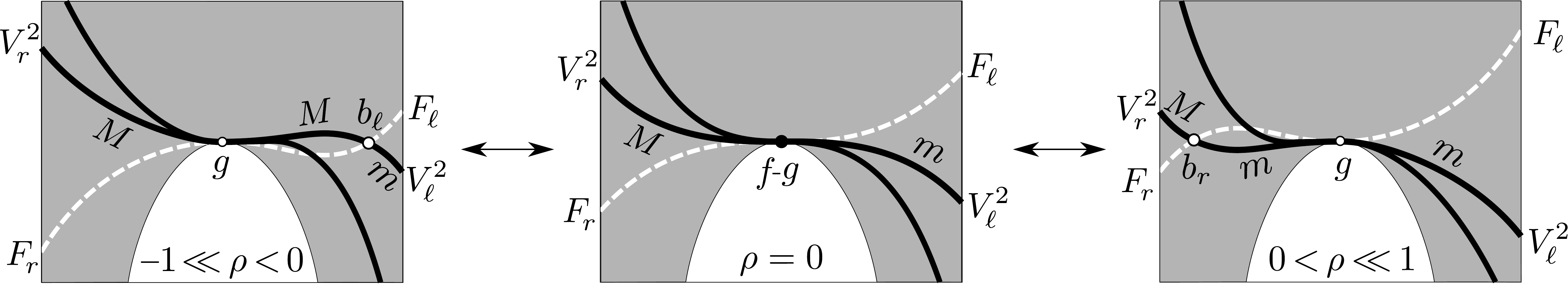}}
\caption{\small A flecgodron transition. }
\label{fig:flecgodron-transition}
\end{figure}

In Figure~\ref{fig:flecgodron-transition}, a left biflecnode $b_\ell$ approaches $g$
as $\rho\to 0^-$, it coincides with $g$ when $\rho=0$ and, as $\rho$ is growing, this point 
leaves $g$ as a right biflecnode $b_r$. By Remark\,\ref{left-right-orient-F} 
and Proposition~\ref{prop:left-righ_Vcurve-at-g}, a biflecnode near $g$ 
is the intersection point of the flecnodal curve $F$ with the tangential component $V^2$ 
of the V-curve ($V^2$ has the same local orientation right-to-left as $F$ near $g$). 

\section{Further investigations}\label{s:further}
In \S\ref{s:cod1} we have formally investigated the `transitional moment' $M_0$ of the families
studied, and explained the transitions by means of examples, but we reserve for further work a formal classification of the families themselves, including the behaviour of the V-curve relative to the flecnodal
curve during these transitions. Global results about the V-curve on a compact surface
are also for further investigation. We do not know whether there are interesting
affinely invariant generalizations, say to
contact of surfaces with conics in their tangent planes. There are also  questions
concerning the symmetry sets or medial axes of the families of curves obtained as plane
sections of a smooth surface parallel to the tangent plane (as mentioned in the Introduction);
these will be investigated elsewhere.

\medskip\noindent
{\sc Acknowledgements} The first and third authors acknowledge support from
the Research Centre in Mathematics and 
Modelling at the University of Liverpool (spring 2008) and by the
Engineering and Physical Sciences Research Council
(autumn 2008, grant number EP/G000786/1) when this problem was originally
studied. The second author acknowledges support from Liverpool Hope University to deliver a talk on this subject at the 6th International Workshop on Singularities in 
Generic Geometry and its Applications, Valencia, Oct 2019.
 The third author acknowledges support from 
Laboratory Solomon Lefschetz UMI2001 CNRS, 
Universidad Nacional Autonoma de M\'{e}xico.



\begin{thebibliography}{99}
\bibitem{BGM} T.Banchoff, T.Gaffney and C.McCrory, {\em Cusps of
Gauss Mappings}, Pitman Advanced Publishing Program 1982.
\bibitem{Bruce} J.W.Bruce, `Lines, circles, focal and symmetry sets', {\em Math.\ Proc.\ Cambridge Philos.\
Soc.}\ 118 (1995), 411--436.
\bibitem{BGsymmsets}J.W.Bruce and P.J.Giblin `Growth, motion and one-parameter families of symmetry sets', {\em Proc. Royal Soc. Edinburgh} 104A (1986), 179--204. 
\bibitem{BGT1} J.W.Bruce, P.J.Giblin and F.Tari, `Parabolic curves of evolving surfaces,'
{\em Int. J. Computer Vision} 17 (1996), 291--306.
\bibitem{BGT2} J.W.Bruce, P.J.Giblin and F.Tari, `Families of surfaces: height functions, Gauss maps and duals', in {\em
Real and Complex Singularities, W.L.Marar(ed.), Pitman Research Notes in Mathematics}, Vol. 333 (1995), 148--178.
\bibitem{BGT3} J.W.Bruce, P.J.Giblin and F.Tari, `Families of surfaces: focal sets, ridges and umbilics', {\em Math. Proc. Camb. Phil. Soc.} 125 (1999), 243-268. 
\bibitem{diatta-giblin} A.Diatta and P.J.Giblin, `Vertices and inflexions of plane sections of surfaces in $\mathbb R^3$', {\em Trends in Mathematics, Real
and Complex Singularities}, Birkh\"{a}user (2006), 71--97.
\bibitem{dimca-gibson} A.Dimca and C.G.Gibson, `On contact germs from the plane to the plane',
{\em Proc.\ Symp.\ in Pure Math.\ } 40.1 (1983), 277--282.
\bibitem{Fukui} T. Fukui, M. Hasegawa and K. Nakagawa, `Contact of a regular surface in Euclidean 3-space with cylinders and cubic binary differential equations',   {\em J. Math. Soc. Japan} 69 (2017), 819-847.
\bibitem{Mumford} P.L.Hallinan, G.G.Gordon, A.L.Yuille, P.Giblin and D.Mumford, {\em Two-and Three-Dimensional Patterns of the Face}, A.K.Peters 1999. 
\bibitem{Shyuichi-et-al} S.Izumiya, M. del C. Romero-Fuster, M.A.S.Ruas and
F.Tari {\em Differential Geometry from a Singularity Theory Viewpoint}, 
World Scientific Pub.\ Co.\ 2015.
\bibitem{Kazarian-Uribe} M. Kazarian, R. Uribe-Vargas,
 `Characteristic Points, Fundamental Cubic Form
    and Euler Characteristic of Projective Surfaces',
{\em Moscow Math.\ J.}  20 (2020), 511--530. 
\bibitem{solidshape} J.J.Koenderink, {\em Solid Shape}, M.I.T. Press 1990.
\bibitem{Montaldi} James A. Montaldi, `On contact between submanifolds', {\em Michigan Math.\ J.}
33 (1986), 195--199.
\bibitem{Montaldi2} James A. Montaldi, `Surfaces in 3-space and their contact with circles', {\em J. Diff.\ Geom.} 23 (1986), 109--126.
\bibitem{Porteous} Ian R. Porteous, {\em Geometric differentiation}, Cambridge University Press, second edition 2001.
\bibitem{ricardo} R. Uribe-Vargas, `A Projective Invariant for Swallowtails and Godrons,
and Global Theorems on the Flecnodal Curve',
{\em Moscow Math.\ J.} \textbf{6} (2006) 731--768. 


\bibitem {Uribeevolution} R. Uribe-Vargas, {\em Surface Evolution, Implicit 
Differential Equations and Pairs of Legendrian Fibrations}, Preprint (2002). 
An improved version has been submitted for publication (2020). 

\bibitem {Uribeinvariant} R. Uribe-Vargas,
{\em On Projective Umbilics: a Geometric Invariant and an Index}.
Journal of Singularities \textbf{17}, Worldwide Center of Mathematics, 
LLC (2018) 81-90. DOI 10.5427/jsing.2018.17e

\medskip\noindent
Peter Giblin, Department of Mathematical Sciences, The University of Liverpool,
Liverpool L69 7ZL, England. Email pjgiblin@liv.ac.uk

\smallskip\noindent
Graham Reeve,  Department of Mathematics and Computer Science, Liverpool Hope University, Liverpol L16 9JD, UK, email reeveg@hope.ac.uk

\smallskip\noindent
Ricardo Uribe-Vargas, Institut de Mathématiques de Bourgogne, UMR 5584, CNRS, 
Université Bourgogne Franche-Comté, F-21000 Dijon, France. 


\end{thebibliography}
\end{document}